\documentclass[10pt]{article}
\usepackage[utf8]{inputenc}
\usepackage[letterpaper,top=1in,bottom=1in,left=1.2in,right=1.2in]{geometry}
\usepackage[english]{babel}
\usepackage{amsmath}
\usepackage{amsthm}
\usepackage{bm}
\usepackage{amssymb}
\usepackage{dsfont}
\usepackage{commath}
\usepackage{mathtools}
\usepackage{float}
\usepackage{algorithm}
\usepackage{algpseudocode}

\usepackage[numbers]{natbib}
\usepackage{graphicx}
\usepackage{enumitem}

\usepackage{tikz}
\usetikzlibrary{positioning}

\newcommand{\eps}{\varepsilon}

\newcommand{\cC}{\mathcal{C}}
\newcommand{\cD}{\mathcal{D}}

\newcommand{\cG}{\mathcal{G}}

\newcommand{\cK}{\mathcal{K}}
\newcommand{\cL}{\mathcal{L}}
\newcommand{\cM}{\mathcal{M}}

\newcommand{\cP}{\mathcal{P}}

\newcommand{\cR}{\mathcal{R}}

\newcommand{\cT}{\mathcal{T}}

\newcommand{\cY}{\mathcal{Y}}
\newcommand{\cZ}{\mathcal{Z}}

\newcommand{\EE}{\mathbb{E}}

\newcommand{\NN}{\mathbb{N}}

\newcommand{\PP}{\mathbb{P}}

\newcommand{\RR}{\mathbb{R}}

\newcommand{\ZZ}{\mathbb{Z}}

\newcommand{\bone}{\mathbf{1}}
\newcommand{\indep}{\perp \!\!\! \perp}

\theoremstyle{plain}
\newtheorem{theorem}{Theorem}[section]
\newtheorem{corollary}[theorem]{Corollary}
\newtheorem{lemma}[theorem]{Lemma}
\newtheorem{proposition}[theorem]{Proposition}

\newtheorem{assumption}[theorem]{Assumption}
\newtheorem{example}[theorem]{Example}

\theoremstyle{definition}
\newtheorem{remark}[theorem]{Remark}
\newtheorem{definition}[theorem]{Definition}

\numberwithin{equation}{section}

\usepackage{fancyhdr}
\usepackage{hyperref}
\title{Particle Systems with Local Interactions via Hitting Times and Cascades on Graphs\footnote{Research of QY was
partially supported by the National Science Foundation grant DMS 2406762.}}
\author{Yucheng Guo and Qinxin Yan}
\author{Yucheng Guo\footnote{Department of Operations Research and Financial
Engineering, Princeton University, Princeton, NJ, 08540, USA, email: 
{\tt yg7348@princeton.edu }. }
\and Qinxin Yan\footnote{Program of Applied and Computational
Mathematics, Princeton University, Princeton, NJ, 08540, USA, email: 
{\tt qy3953@princeton.edu}. }}

\date{}
\usepackage{oubraces}

\begin{document}

\maketitle

\vspace{0mm}
\begin{abstract}
We introduce a family of particle systems on sparse graphs where local interactions occur via hitting times, providing a dynamic and tractable model for default cascades in large sparsely-connected financial networks. Building on the framework of \cite{lacker2022localweakconvergencesparse}, we extend convergence theory to systems with singular interactions, capturing the abrupt and discontinuous nature of systemic events. We establish conditions for well-posedness through a minimality principle and connect fragility to dynamic percolation thresholds. Our analysis demonstrates continuity of the joint law of defaults with respect to local graph convergence, establishes convergence of empirical distributions, and characterizes the default time distribution in tree-like networks. This framework offers a rigorous and flexible foundation for modeling systemic risk in evolving financial systems, featuring continuous-time dynamics, heterogeneous and local interactions, and instantaneous default cascades.
\end{abstract}
\vspace{1mm}

\vspace{3pt}

\def\d{\mathrm{d}}
\section{Introduction}\label{sec:intro}
The global financial crisis of 2008 and the more recent bank failures in 2023 highlighted how interconnections among financial institutions can amplify idiosyncratic shocks into systemic events. This has made the analysis of default cascades in financial networks a central problem in systemic risk. Many existing models describe interbank contagion through static clearing mechanisms or discrete-time dynamics. Although these approaches have yielded important insights, they are less well suited to capturing the continuous-time, discontinuous, and self-exciting nature of contagion generated by sequential defaults. In this paper, we introduce a continuous-time locally interacting particle system for default cascades on networks, providing a tractable framework for the analysis of systemic fragility and resilience.

We propose to study the following interacting particle system on a countable and locally finite graph $G=(V,E)$:
\begin{equation}\label{eq:1.1}
\begin{aligned}
X_v(t)&=x_v+Z_v(t)-\sum_{u\in N_G^-(v)}c_{uv}\bone_{\{u\in D_t\}},\quad v\in V\\
D_t&=\big\{v\in V:\,\inf_{s\in[0,t]}X_v(s)\leq0\big\}.    
\end{aligned}    
\end{equation}
It can be seen as a stylized model for an inter-connected network of banks with mutual lending, in which the default of a bank leads to immediate losses to its creditors. In this system, each vertex $v\in V$ represents a particle (financial institution, such as a bank) with initial health (asset value) $x_v \in [0, \infty)$, $(Z_v)_{v\in V}$ are a pre-determined collection of $V$-indexed c\`adlàg stochastic processes (e.g. Brownian motions or more general Lévy processes) which represent the noise driving the dynamics, $N_G^-(v):=\{u\in V:(u,v)\in E\}$ denotes the in-neighborhood of $v$, and $c_{uv}\ge0$ is the loss suffered by particle $v$ if particle $u$ dies (or defaults)\footnote{Rigorously speaking, the loss triggered by the default of a bank depends on multiple factors such as interbank lending relationships, asset and liability structure, counterparty risk in derivatives, clearing and settlement mechanism, investor sentiment, and regulatory scrutiny. We choose to set the default loss as a pre-determined constant to abstract away from the idiosyncratic factors that determine loss severity and concentrate on the structural dynamics of systemic risk transmission.}, that is, particle $u$’s health falls to $0$. The set-valued process $t\mapsto D_t$ records the set of banks which default no later than time $t$.

\begin{remark}
In this model, when a bank defaults, it triggers immediate losses to its neighbors, after which it cannot have a second impact on the system. In that sense, the defaulted bank can be treated as ``removed" from the network, even though the network structure remains the same and the corresponding driving noise $Z_v(t)$ continues to evolve in a purely imaginary manner.
\end{remark}

The above system can be formulated as a \textit{pathwise} fixed-point problem. Given a realization of the graph $G$ and the driving noises $(Z_v)_{v\in V}$, we define the operator $\Gamma$, which maps a set-valued process $D=(D_t)_{t\ge 0}$ to another set-valued process $\Gamma[D]=(\Gamma[D]_t)_{t\ge 0}$ via
\begin{equation*}\begin{aligned}
\Gamma[D]_t:=\Big\{v\in V:\,\inf_{s\in[0,t]}\Big(x_v+Z_v(t)-\sum_{u\in N_G^-(v)}c_{uv}\bone_{\{u\in D_t\}}\Big)\leq0\Big\}.
\end{aligned}\end{equation*}
Then $(X,D)$ is a solution to the system if and only if equation~\eqref{eq:1.1} holds and $D$ satisfies the fixed-point condition $\Gamma[D]=D$.

We first point out that equations \eqref{eq:1.1} do not uniquely pin down a dynamic, even in the simplest case:
\begin{example}
Let $G:=\Big(V=\{1,2\},E=\{(1,2),(2,1)\}\Big)$, $x_1=x_2=1$ and $c_{12}=c_{21}=1$. Then,
\begin{enumerate}[label=(\alph*)]
    \item there exists a solution $(X,D)$ such that $D_t=\emptyset$ for all $t<\tau$, where $\tau:=\inf\{t\ge0:\,Z_1(t)\leq-1\,\,\text{or}\,\,Z_2(t)\leq-1\}$.
    \item Let $\tilde D_t=\{1,2\}$ for all $t\ge0$ and $\tilde X_i(t)= Z_i(t)$, $t\ge0$ for $i=1,2$.
Then $(\tilde X, \tilde D)$ is also a solution.
\end{enumerate}
\end{example}

The issue with the second pathological solution $(\tilde X,\tilde D)$ is that the banks default \emph{when they do not have to}. This leads to a self-justifying group of defaults, where each unnecessary default is justified by the others.  To exclude such behaviors, we would like to seek solutions that satisfy the following \emph{physicality condition}.

\begin{definition}\label{def:physical}
A solution $(X, D)$ is said to be \emph{physical}, if 
\begin{enumerate}[label=(\alph*)]
    \item the map $t\mapsto D_t$ is right-continuous, and
    \item for any $t$ such that $D_t\neq D_{t-}$, it holds that $D_t=D_t^{(\infty)}$, where the latter is given by the following iterative construction:
    \begin{equation*}\begin{aligned}
    D_t^{(0)}&:=D_{t-}\cup\Big\{v\in V:\,x_v+Z_v(t)-\sum_{u\in N_G^-(v)}c_{uv}\bone_{\{u\in D_{t-}\}}\leq0\Big\},\\
    D_t^{(N+1)}&:=D_t^{(N)}\cup\Big\{v\in V:\,x_v+Z_v(t)-\sum_{u\in N_G^-(v)}c_{uv}\bone_{\{u\in D_t^{(N)}\}}\leq0\Big\},\quad N\ge0,\\
    D_t^{(\infty)}&:=\bigcup_{N=0}^\infty D_t^{(N)}.
    \end{aligned}\end{equation*}
\end{enumerate}
\end{definition}
\begin{remark}
\begin{itemize}
    \item[(a)] In the setting of systemic risks, the default cascade process $(D_t^{(N)})_{N\ge0}$ is the same as the one studied by Amini, Cont and Minca in \cite{amini2016resilience}. The idea of obtaining the smallest default set by an iteration scheme can be dated back to the \emph{fictitious default algorithm} proposed by Eisenberg and Noe in \cite{eisenberg2001systemic}. 
    \item[(b)] Physical solutions describe the systems in which the defaults can be fully ordered and attributed. Indeed, as $t\mapsto D_t$ is a right continuous set valued process, all the increments come from the left jumps $D_t\setminus D_{t-}$, which has a hierarchical structure characterized by the sequence $(D_t^{(N)})_{N\ge0}$. This hierarchical structure will be further discussed and exploited in Section \ref{sec:Trees}. 
    \item[(c)] At the initial time $t=0$, we do not rule out potential discontinuities of the system (such as the occurrence of defaults), hence we need to extend the time range for the dynamics to $[-1,\infty)$ by setting 
    \begin{equation*}\begin{aligned}
    X_v(t)=x_v,\quad D_t=\emptyset,\quad Z_v(t)=0,\quad t\in[-1,0).    
    \end{aligned}\end{equation*}
    With this extension, it becomes meaningful to consider the left limits of the dynamics at $t=0$, which encode the initial conditions for the dynamics. Physical solution $(X, D)$, if it exists, will have all of its paths $X_v(\cdot)$ and $\bone_{\{v\in D_\cdot\}}$ belong to the space $\cD:=D([-1,\infty))$. A brief introduction to the space $\cD$ as well as the $M_1$ topology on it is given in Appendix \ref{sec:cD}.
\end{itemize} 

% \noindent (3) If we define $D_t^{(-1)}:=D_{t-}$ by convention, then the inductive relation
% \begin{equation*}\begin{aligned}
% D_t^{(N+1)}:=D_t^{(N)}\cup\Big\{v\in V:\,x_v+B_v(t)-\sum_{u\in N_G^-(v)}c_{uv}\bone_{\{u\in D_t^{(N)}\}}\leq0\Big\} 
% \end{aligned}\end{equation*}
% also holds for $N=-1$.
\end{remark}

Intuitively, at times of unavoidable defaults, physical solutions have \emph{minimal} jumps. This observation leads us to consider the so-called \emph{minimal solutions}, which can be obtained by a straightforward iteration argument leveraging the monotonicity of  the map $\Gamma$.
\begin{definition}
A solution $(X,D)$ is said to be \emph{minimal}, if for any other solution $(\tilde X,\tilde D)$ it holds that
\begin{equation*}\begin{aligned}
D_t\subset\tilde D_t,\quad\forall\, t\ge0. 
\end{aligned}\end{equation*}
\end{definition}

It follows immediately from the definition that the minimal solution, if it exists, must be unique.The idea of characterizing physical solutions through minimal solutions goes back to \cite{cuchiero2020propagation}.

Establishing the well-posedness of the system becomes challenging when the underlying graph $G$ is infinite. First, if the noise processes $(Z_v)_{v\in V}$ are i.i.d., the zero-hitting times are typically dense in the time axis $(0,\infty)$. As a result, the system may lack a strict separation between continuous evolution and jump regimes. Second, without further assumptions,  stability cannot be guaranteed even on arbitrarily short time intervals, as the following example demonstrates.

\begin{example}\label{ex:Minimal_Not_Physical}
We take $V:=\NN=\{0,1,2,...\}$ and $E:=\{(n,n+1):\,n\in\NN\}$ with $c_{n+1,n}=1$. Let $x_n:=\frac{1}{2^{n+1}}$ for all $n\ge0$ and let $(Z_n=B_n)_{n\ge0}$ be independent standard Brownian motions. Then any solution $(X,D)$ to equations \eqref{eq:1.1} satisfies $D_{0+}=V=\NN$ with probability $1$. In particular, there is no physical solution to equations \eqref{eq:1.1} and with the initial conditions given above.
\end{example}

To ensure the well-posedness of the system, we introduce a set of \emph{robustness} conditions on the problem configuration.
\begin{definition}
The problem configuration $(G,c,x,Z)$ is said to drive \emph{$\delta$-robust systems} if, with probability one, it holds that for every $t\ge0$, the set
\begin{equation*}\begin{aligned}
\Big\{v\in V:\exists s\in[t,t+\delta],\text{ s.t. }x_v+Z_v(s-)\ge0\quad\text{and} \quad x_v+Z_v(s)\leq\sum_{u\in N_G^-(v)}c_{uv}\Big\}\,
\end{aligned}\end{equation*}
contains no infinite weakly connected component. The problem configuration $(G,c,x,Z)$ is said to drive \emph{robust systems} if it drives $\delta$-robust systems with $\delta=0$.
\end{definition}

A collection of sufficient conditions for $\delta$-robustness are provided in Lemma \ref{lem:robust1}. As it turns out, $\delta$-robustness is sufficient for the well-posedness of physical solutions to equations~\eqref{eq:1.1}. Before stating our main results, we state the general assumptions imposed on the driving noise component $Z$ throughout the paper:
\begin{assumption}\label{ass:Main}
\begin{enumerate}[label=(\alph*)]
    \item $Z$ has no upward jumps:
    \begin{equation*}
    Z(t)-Z(t-)\leq0,\quad\forall t\ge0.    
    \end{equation*}
    \item $Z$ satisfies the \emph{downward crossing property}: for any fixed stopping time $\tau$,
    \begin{equation*}\begin{aligned}
    \PP\Big[\forall h>0,\,\inf_{s\in[0,h]}(Z(\tau+s)-Z(\tau))<0\Big]=1.    
    \end{aligned}\end{equation*}
\end{enumerate}
\end{assumption}
\begin{remark}
The functional that takes a càdlàg process to its zero-hitting time is continuous at those processes satisfying the downward crossing property (see e.g. \cite[Proposition 5.8]{DELARUE2015Spike}). Formally speaking, any càdlàg stochastic process with a Brownian component satisfies the above downward crossing property.  The scope of financial models encompassed by our framework is discussed further in Lemma~\ref{lem:robust1}, Remark~\ref{rem:robust1}, and Section~\ref{sec:MajorBanks}.
\end{remark}
\begin{theorem}[Well-Posedness]\label{thm-main-1}
For each problem configuration $(G,c,x,Z)\in\cG_*[\RR\times\cD]$ that drives $\delta$-robust systems, there exists a unique physical solution to the equations \eqref{eq:1.1}. Moreover, this solution coincides with the minimal solution. 
\end{theorem}

After establishing well-posedness, we proceed to answer the following two approximation questions:
\begin{itemize}
    \item Suppose an infinite graph can be approximated by a sequence of finite graphs in a suitable sense. Does it imply the convergence of physical solutions?
    \item If, in addition, the involved graphs have symmetric structures, can we analyze the empirical distribution of the solution paths and the default times
    \begin{equation*}\begin{aligned}
    \mu^n:=\frac{1}{|G_n|}\sum_{v\in G_n}\delta_{(X_v^n,\tau_v^n)}    
    \end{aligned}\end{equation*}
    by focusing on a \emph{representative} vertex?
\end{itemize}

The paper \cite{lacker2022localweakconvergencesparse} by Lacker, Ramanan and Wu provides a theoretical framework for studying interacting particle systems on large sparse graphs with local interaction, which is well-suited for answering the above two approximation questions. In this framework, both the problem configurations and the solutions are encoded by (random) \emph{marked graphs}, on the space of which the topology of local convergence can be endowed with. It turns out that \emph{local convergence} is a proper notion of the approximation in the first question, and \emph{local weak convergence} provides a natural characterization of symmetry for the second question. These will be explained in more detail in Section \ref{sec:Local_Conv} and Section \ref{sec:Local_Weak}. One of our main contributions is to extend the general theory presented in \cite{lacker2022localweakconvergencesparse} to particle systems with singular interaction via hitting times.

\begin{theorem}[Convergence of Physical Solutions]\label{thm-main-2}
Let $(G_n,c^n,x^n,Z^n)$ be a sequence of problem configurations that drive robust systems and that $\cL(G_n,c^n,x^n,Z^n)\to\cL(G,c,x,Z)$ in $\cP(\cG_*[\RR\times\cD])$ as $n\to\infty$, where the limit configuration $(G,c,x,Z)$ drives $\delta$-robust systems. Then  $
\mathcal{L}(G_n, X^n, D^n) \to \mathcal{L}(G, X, D) \, \text{in } \mathcal{P}(\mathcal{G}_*[\mathcal{D}^2]),$
where $ (X^n, D^n) $ and $ (X, D) $ denote the unique physical solutions associated with $(G_n,c^n,x^n,Z^n)$ and $(G,c,x,Z)$, respectively.
\end{theorem}

\begin{theorem}[Convergence of Empirical Measures]\label{thm-main-3}
Let $(G_n,c^n,x^n,Z^n)$ be a sequence of problem configurations such that each each $G_n$ is finite, and suppose that $(G_n,c^n,x^n,Z^n)$ converges in probability in the local weak sense to a random element $(G,c,x,Z)\in \cG_*[\RR\times\cD]$, which drives $\delta$-robust systems. Then the corresponding solutions $(G_n,X^n,D^n)$ converges in probability in the local weak sense to $(G,X,D)$. In particular, the empirical measure $\mu^n$ converges in probability to $\cL(X_o,\tau_o)$, where $o$ is the root of $G$. 
\end{theorem}

The main technical challenge in proving Theorems \ref{thm-main-1} and \ref{thm-main-2} is to qualitatively control the “propagation speed’’ of default cascades. Although the interactions in the system \eqref{eq:1.1} are local, it remains possible in principle for a cascade to spread instantaneously across a large portion of the network, exacerbating the discontinuity of the system under consideration. This phenomenon does not arise in \cite{lacker2022localweakconvergencesparse} and therefore falls outside the scope of the methods developed there. To overcome this difficulty,  our approach is to analyze carefully the causal structure of defaults and thereby re-establish locality of the dynamics, which is carried out in Section~\ref{sec:Trees} using tools from dynamic percolation theory. This step quantifies how the network topology constrains the growth of the default sets and forms the core technical contribution of the paper. In contrast, the proof of Theorem \ref{thm-main-3} is more direct than that of \cite[Theorem 3.7]{lacker2022localweakconvergencesparse}, since the noise processes $(Z_v)_{v\in V}$ may be incorporated into the inputs of the dynamics.

The remainder of this section provides a brief review of related literature. In Section~\ref{sec:def}, we introduce the main definitions and notations used throughout the paper. Section~\ref{sec:three} presents the main results and their proofs. The proof of Theorem~\ref{thm-main-1} is established in Section \ref{sec:physical} and \ref{sec:robustness}. Theorem~\ref{thm-main-2} is proved in Section~\ref{sec:proof of main2}, and Theorem~\ref{thm-main-3} in Section~\ref{sec:proof of main3}. In Section~\ref{sec:delayedmodel}, we demonstrate the connection between the model studied in this paper and other models of systemic risks in which losses are not realized immediately. In Section~\ref{sec:MajorBanks}, we discuss potential generalization of our framework to include both major banks and local banks. Additional technical proofs are provided in the Appendices.

\subsection{Related literature.}
Modeling the propagation of defaults in financial networks has emerged as a foundational problem in the study of systemic risk. Models for default cascades aim to capture how the failure of one financial institution can trigger a sequence of failures throughout the network—a process akin to contagion. A foundational framework is laid by the Eisenberg–Noe clearing payment model \cite{eisenberg2001systemic}, to which the idea of analyzing the default set via a hierarchical structure can be dated back to. 

A central question that remains challenging is how the structure of financial networks affects their resilience to noises and shocks. Recent refinements in economic and network theory, including \cite{acemoglu2015systemic} and \cite{battiston2012debt}, further demonstrated how interconnections can have both stabilizing and destabilizing effects. The pioneering work \cite{amini2016resilience} proposed a structural model of financial contagion on random graphs, introducing a quantitative measure for resilience in large, sparsely connected financial networks. Their work provides explicit conditions under which a network is resilient to small shocks—meaning a vanishing fraction of defaults in the limit of large system size—based on the degree distribution and exposure profile. By characterizing threshold behavior for systemic events in terms of the fixed-point structure of default probabilities, their model offers both analytical tractability and valuable insight into how network topology governs the likelihood of large-scale cascades. We also refer the reader to \cite{amini2016resilience} for a overview of previous researches in this direction. Partially due to the consideration of analytical tractability, the models that have been most studied are either static or have discrete-time dynamics and are mostly set up on finite networks. 

To manage the complexity of analyzing continuous dynamics on large-scale financial networks, recent work, including \cite{capponi2015systemic,carmona2013mean}, has adopted \textit{mean-field} formulations which aim to describe the macroscopic behavior of the system as the number of nodes tends to infinity by focusing on a representative particle, which is a phenomenon referred to as propagation of chaos, under assumptions of complete symmetry and weak interactions. This approach has proven fruitful in both stochastic control and game-theoretic settings. The pioneering works \cite{delarue2015global,DELARUE2015Spike} introduced the instantaneous contagion and the notion of physical solutions to mean-field particle systems and established the corresponding propagation of chaos. Also revealed are their deep connections, as a probabilistic method, to the supercooled Stefan problem, a singular free boundary problem, the well-posedness of which has remained challenging for decades. Following this framework, \cite{dns} established the global-in-time well-posedness of physical solutions and studied their regularities, and \cite{cuchiero2020propagation} proposed the notion of minimal solutions, pointed out their connections to physical solutions, and established the propagation of minimality. These models and their variants have been further developed and applied to the study of financial networks and neuronal networks as in \cite{hambly_spde_2019,NaSh1,NaSh2,Hambly2022Elastic,baker2025localtime,hambly2023contagiouscommon,hambly2023controlcontagion,ledger2024transport,baker2024twophase,Cuchiero2024bailout,erhan2025systemic}.

While mean-field models have proven instrumental in describing the asymptotic behavior of large populations, they are often based on assumptions of full homogeneity and dense connectivity. The articles \cite{NaSh2,Feinstein2023heterogeneous} incorporated heterogeneity into the mean-field framework by constructing and analyzing multi-type particle systems, with the interactions between the types given by probability kernels. Within the types, the particles are still exchangeably distributed, densely connected, and each individual particle interacts with the whole distribution and has negligible effect on the rest of the system. Such assumptions can be too restrictive in the study of financial networks, which are typically sparsely connected and exhibit local heterogeneous interactions driven by direct bilateral exposures, as illustrated by empirical evidence given in \cite{Boss2004network,Cont_Moussa_Santos_2013,Craig2014network}. To address this gap, we adopt the general theory developed in \cite{lacker2022localweakconvergencesparse} of the local convergence for networks of interacting stochastic processes, which is particularly well-suited for modeling cascades in networks, where contagion spreads through direct neighboring exposure rather than global interaction. It is worth emphasizing that this form of large sparse graph limits is not covered by the theory of mean-field or graphon limits, as the correlation between the state processes of any fixed pair of vertices will typically not vanish in the limit.

% In summary, this paper develops a general mathematical framework for analyzing systemic risk in large, sparsely connected networks with local interactions and instantaneous default cascades. We formulate the system as an interacting particle system on a sparse graph and establish its well-posedness by combining a percolation result from graph theory with a monotonicity property of the default set, ensuring the cascade process remains locally finite. From a technical standpoint, the key difficulty arises from the singular and path-dependent nature of the interaction, driven by hitting times. We address this by extending the local convergence framework of \cite{lacker2022localweakconvergencesparse} to systems with singular interactions, and we develop a coupling-based argument to prove continuity with respect to local convergence of graphs. Our analysis also discuss the connection between the growth of the default sets and the underlying network topology, which may be of independent interest.

% \red{Several limitations remain. The current model represents the state of each node by a one-dimensional process without incorporating balance sheet details. The network is static and defaults are irreversible, whereas real applications, for example real financial systems,  involve dynamic link formation and entry of new participants. Moreover, we do not model endogenous strategic behavior in network formation. These directions are left for future work.
% }
 
In summary, this paper contributes to the literature on mathematical finance by developing a novel continuous-time framework for analyzing systemic risk in large, sparsely connected financial networks, with particular emphasis on local interactions and instantaneous default cascades. From the perspective of interacting particle systems, we extend the general convergence results of \cite{lacker2022localweakconvergencesparse} to models with singular interactions arising from hitting times. A key element of the analysis is the use of percolation theory to rigorously quantify the relationship between network topology and the growth of default sets, which constitutes an independent contribution of interest.

The present framework has several natural limitations. First, we model the health of a bank as a single real-valued process, without incorporating the detailed structure of its balance sheet. In practice, such structural features play an important role in determining the magnitude of losses borne by creditors in the event of default. Second, the network is taken to be static and each institution may default only once, whereas in reality the web of mutual exposures evolves over time, new institutions may enter, and defaulted ones are typically replaced. Finally, we do not account for the strategic behavior of institutions in shaping network structure and exposures, which would introduce an additional layer of complexity. Incorporating these features lies beyond the scope of the present work but represents an important direction for future study.

\section{Preliminaries and Notations}
\label{sec:def}
\subsection{Elementary notations}
For a random element $X$, let $\cL(X)$ denote its distribution measure. For any topological spaces $\cY,\cZ$, $C(\cY \to \cZ)$ is the set of all continuous maps 
$f:\cY \to \cZ$, $C(\cY):=C(\cY \to \RR)$, and $C_b(\cY)$ denotes the set of bounded continuous functions. We use $\cC:=C([-1,\infty))$ to denote the space of continuous functions defined on $[-1,\infty)$, endowed with the topology of uniform convergence on compact sets. For any subset $E$ of $\cY$, we denote the restriction of $f$ on $E$ by $f|_{E}$. The indicator function of a set (event) $A$ will be denoted by $\bone_A$.

We denote by $\cD:=D([-1,\infty))$ the space of functions $f:[-1,\infty)\to\RR$ that are right-continuous at all $t\in[-1,\infty)$ and have left limits at all $t\in(-1,\infty)$. The space $\cD$ will be endowed with the Skorokhod $M_1$ topology throughout the paper. A brief technical introduction to the $M_1$ topology can be found in Appendix \ref{sec:cD}.

\subsection{Local convergence of marked graphs}\label{sec:Local_Conv}
For this part, we follow the general framework developed in \cite{lacker2022localweakconvergencesparse}, with the main difference being that our graphs will be directed and weighted.
\subsubsection{Directed graphs}
A directed graph can be represented as a pair $G=(V,E)$, where $V$ is a set of vertices and $E\subset\{(u,v)\in V\times V:\,u\neq v\}$ is a set of (directed) edges, given as ordered pairs of vertices. In this paper, we will use $v\in V$ and $v\in G$ interchangeably. For a vertex $v$, the \emph{in-neighborhood} and \emph{out-neighborhood} are defined by
\begin{equation*}\begin{aligned}
N_G^-(v):=\{u\in V:\,(u,v)\in E\},\quad N_G^+(v):=\{u\in V:\,(v,u)\in E\}.     
\end{aligned}\end{equation*}
A directed graph $G$ naturally induces an undirected graph by ignoring the directions of all the edges. The (weak) distance $d_G(u,v)$ between vertices $u,v\in V$ is defined as the length of the shortest path connecting them in the induced undirected graph. The ball of radius $k$ centered at $v\in V$ is defined by 
$B_G(v,k):=\{u\in V:\,d_G(u,v)\leq k\}$.
With a slight abuse of notation, we also define the $k$-enlargement of a set $V_0\subset V$ as
\begin{equation*}\begin{aligned}
B_G(V_0,k):=\bigcup_{v\in V_0}B_G(v,k).
\end{aligned}\end{equation*}
The set of neighbors of a vertex $v$ is defined as $N_G(v):=B_G(v,1)\setminus\{v\}=N_G^-(v)\cup N_G^+(v)$. Similarly, for a subset $V_0\subset V$, the set of neighbors, or the \emph{outer boundary} is given by $\partial^{\mathrm{out}}V_0:=B_G(V_0,1)\setminus V_0$, consisting of all vertices 
 adjacent to at least one vertex in $V_0$, but not in $V_0$ themselves.
In this paper, we do not use any other distances on the directed graph $G$, so we omit the stress of ``weak" in the sequel.

Given a graph $G=(V,E)$ and a subset $V_0$ of its vertex set $V$, the \emph{induced subgraph} on $V_0$ is defined as $G|_{V_0}=(V_0,E\cap (V_0\times V_0))$. When the context is clear, we may refer to this subgraph simply by its vertex set $V_0$ to simplify notations.

\subsubsection{Rooted graphs, isomorphism and the space $\cG_*$} A \emph{rooted graph} $G=(V,E,o)$ is a graph $(V,E)$ with a distinguished vertex $o\in V$, called the \emph{root}. Two rooted graphs $G_i=(V_i,E_i,o_i)$, $i=1,2$ are \emph{isomorphic} if there exists a bijection such that $$
\varphi(o_1) = o_2 \quad \text{and} \quad (u, v) \in E_1 \ \Leftrightarrow\ (\varphi(u), \varphi(v)) \in E_2 \quad \text{for all } u, v \in V_1.
$$ We denote this by $G_1\cong G_2$, and denote by \( I(G_1, G_2) \) the set of all such isomorphisms.

Let $\cG_*$ denote the set of isomorphism classes of connected rooted graphs. A sequence $(G_n)_{n\ge0}$ is said to \emph{converges locally} to $G$ in $\cG_*$ if, for every $k\in\NN$, there exists $N_k\in \NN$ sufficiently large so that $B_{G_n}(o_n,k)\cong B_G(o,k)$ for all $n\ge N_k$. The following metric is compatible with local convergence and renders $\cG_*$ a complete and separable metric space:
\begin{equation*}\begin{aligned}
d_*(G,G'):=\sum_{k=1}^\infty\frac{1}{2^k}\bone_{\{B_G(o,k)\not\cong B_{G'}(o',k)\}}.    
\end{aligned}\end{equation*}
\begin{remark}
If a sequence $ (G_n)_{n \ge 1} $ converges locally to $ G $ in $ \mathcal{G}_* $, then for every $ k \in \mathbb{N} $, there exists $ N_k \in \mathbb{N} $ such that for all $ n \ge N_k $, the balls $ B_{G_n}(o_n, k) $ and $ B_G(o, k) $ are isomorphic as rooted graphs. In particular, we may re-label the vertex set $ V_n $ (via a choice of rooted graph isomorphism) so that
$B_{G_n}(o_n, k) = B_G(o, k)$
holds as subgraphs. In the sequel, we will use this identification whenever it simplifies notation and does not cause confusion.
\end{remark}

\subsubsection{Marked graphs and the space $\cG_*[\cY]$}
For a metric space $(\cY,d)$, a $\cY$-marked rooted weighted graph is a tuple $(G,c,y)$, where $G=(V,E,o)\in\cG_*$, $y=(y_v)_{v\in V}\in\cY^V$ is a vector of marks indexed by $V$, and $c=(c_{uv})_{(u,v)\in E}\in\RR_+^E$ is a vector of weights indexed by $E$. Isomorphisms between rooted weighted graphs are defined analogously to the unmarked case: two $\cY$-marked rooted weighted graphs $(G,c,y)$ and $(G',c',y')$ are \emph{isomorphic} if there exists an isomorphism $\varphi$ from $G$ to $G'$ such that $y_v=y_{\varphi(v)}'$ for any $v\in V$ and that $c_{uv}=c_{\varphi(u)\varphi(v)}'$ for any $u,v\in V$ such that $(u,v)\in E$. We denote by $\cG_*[\cY]$ the set of isomorphism classes of $\cY$-marked graphs. 

A sequence $(G_n,c^n,y^n)$ converges locally to $(G,c,y)$ in $\cG_*[\cY]$ if,  for any $k\in\NN$ and any $\varepsilon>0$, there exists $N_{k,\varepsilon}\in \NN$ sufficiently large so that, for all $n\ge N_{k,\varepsilon}$, $B_{G_n}(o_n,k)\cong B_G(o,k)$, and that there exists an isomorphism $\varphi$ from $B_G(o,k)$ to $B_{G_n}(o_n,k)$ such that, for every $v\in B_G(o,k)$, $d(y_v,y_{\varphi(v)}^n)<\varepsilon$, and $\sum_{u\in N_G^-(v)}|c_{uv}-c_{\varphi(u)\varphi(v)}^n|<\varepsilon$.

The following metric metrizes this convergence and renders $ \mathcal{G}_*[\mathcal{Y}] $ a Polish space whenever $ (\mathcal{Y}, d) $ is Polish:
\begin{equation*}\begin{aligned}
&\quad d_*((G,c,y),(G',c',y'))\\
&:=\sum_{k=1}^\infty 2^{-k}\Bigg(1\wedge\inf_{\varphi\in I(B_G(o,k),B_{G'}(o',k))}\max_{v\in B_G(o,k)}\Big(d(y_v,y_{\varphi(v)}')+\sum_{u\in N_G^-(v)}|c_{uv}-c_{\varphi(u)\varphi(v)}'|\Big)\Bigg).   
\end{aligned}\end{equation*}

In this paper, the metric space $(\cY,d)$ for the marks will always be Polish.

\begin{remark}
\begin{enumerate}[label=(\alph*)]
    \item When the weights are trivial (e.g., $c_{uv} = 1$ for all $(u, v) \in E$), we recover the setting of unweighted graphs as in \cite{lacker2022localweakconvergencesparse}. We will omit the $c$-component when the weights are trivial.

    \item From an abstract perspective, edge weights can be viewed as marks on the set of edges, and in principle, they could take values in a general metric space. However, we do not pursue this level of generality, as the edge weights in this paper are always taken to be non-negative real numbers.

    \item Following the convention in \cite{lacker2022localweakconvergencesparse}, we use vertex marks to represent the initial conditions $(x_v)_{v \in V}$, the driving noises $(Z_v)_{v \in V}$, and the solution trajectories $(X_v)_{v \in V}$.

    % \item[(4)] Given the considerations above, the structure of the marks is more central to our analysis than the structure of the weights. For this reason, we will often omit the weight function from the notation when doing so does not cause confusion. This practice is further justified by the convention in graph theory that the weight function $c$ can be incorporated into the definition of the graph $G$ itself.
\end{enumerate}
\end{remark}

% We obtain the following two Lemmas using the definition of the metric $d_*$ on $\cG_*$ and on $\cG_*[\cY]$.

% \begin{lemma}\label{lem:MarkedGraph_Concatenate}
% The map $(G,y)\mapsto (G,y)$ from $\cG_*\times\cY^V$ to $\cG_*[\cY]$ is continuous, provided that $\cY^V$ is endowed with the product topology.
    
% \end{lemma}

% \begin{lemma}
% Let $f:\prod_{i=1}^m\cY_i\to\cY$ be a continuous map and $(G,y^i)\in\cG_*[\cY_i]$ for $i=1,...,m$. Then the map
% \begin{equation*}\begin{aligned}
% ((G,y^1),...,(G,y^m)\mapsto(G,f(y^1,...,y^m))    
% \end{aligned}\end{equation*}
% is continuous from $\prod_{i=1}^m\cY_i\to\cY$ to $\cG_*[\cY]$.
% \end{lemma}

\begin{lemma}\label{lem:MarkedGraph_ContMap}
Let $\cY$, $\cY'$ be metric spaces, and let  $f:\cY\to\cY'$ be a continuous map. Then the induced map
% \begin{equation*}\begin{aligned}
% (G,c,y)\mapsto(G,c,f(y))    
% \end{aligned}\end{equation*}
\begin{equation*}
\begin{aligned}
(G,c,y)\mapsto(G,c,f(y))      
\end{aligned}
\end{equation*}
is continuous from $\cG_*[\cY]$ to $\cG_*[\cY']$.
\end{lemma}
\begin{proof}
This follows directly from the definition of the metric $d_*$. 
% and the continuity of $f$.
% Indeed, continuity of $f$ implies that small changes in $y$ (in the $\mathcal{Y}$-metric) lead to small changes in $f(y)$ (in the $\mathcal{Y}'$-metric), uniformly on finite neighborhoods, which is precisely the structure encoded by $d_*$.
\end{proof}

\subsubsection{Compactness in $\cG_*[\cY]$}
We next state a technical lemma that provides a general method for establishing compactness in $ \mathcal{G}_*[\mathcal{Y}' \times \mathcal{Y}] $ by leveraging compactness criteria in the underlying space $ \mathcal{Y} $. The proof of Lemma~\ref{lem:MarkedGraph_Comp} is in Section~\ref{sec:app-technical}.

\begin{lemma}\label{lem:MarkedGraph_Comp}
Let $K_m$ be a non-decreasing sequence of compact subsets of $\cY$, and let $\cK_0$ be a compact subset of $\cG_*[\cY']$. Then the set
\begin{equation*}\begin{aligned}
\cK:=\Big\{(G,c,y',y)\in\cG_*[\cY'\times\cY]:\,(G,c,y')\in\cK_0,\, y_v\in K_m\,\,\text{for all}\,\, v\in B_G(o,m),\,\forall m\in\NN\Big\}    
\end{aligned}\end{equation*}
is a compact subset of $\cG_*[\cY'\times\cY]$.
\end{lemma}

% \subsubsection{Examples}
% We outline a few examples of models that can be studied within the current framework.
% \begin{itemize}
%     \item Let $G_n$ be the cycle on $n$ vertices. As $n$ increases, $G_n$ converges locally to $(\mathbb{Z},0)$, which is the two-way infinite line graph rooted at $0$.
%     \item The Erdös-Rényi graph $G_n\thicksim \cG(n,p_n)$ with $\lim_{n\rightarrow \infty} np_n=\theta\in (0,\infty)$. The local limit is the Galton–Watson tree with Poisson$(\theta)$ offspring distribution, rooted at the progenitor.
% \item Let $G_n$ be the $d$-dimensional discrete torus graph on $n^d$ vertices, i.e., the vertex set is $(\mathbb{Z}/n\mathbb{Z})^d$ with edges between nearest neighbors. Then
%     $$
%     G_n \longrightarrow (\mathbb{Z}^d, 0),
%     $$
%     the infinite $d$-dimensional lattice rooted at the origin.

%     \item Let $G_n$ be a uniform random $d$-regular graph on $n$ vertices for fixed $d \ge 3$. Then, as $n \to \infty$,
%     $$
%     G_n \longrightarrow \mathbb{T}_d,
%     $$
%     where $\mathbb{T}_d$ is the infinite $d$-regular tree rooted at an arbitrary vertex.

%     \item Let $G_n$ be a geometric random graph with $n$ vertices uniformly distributed in the unit cube $[0,1]^d$, and edges drawn between pairs at distance less than $r_n$, with
%     $$
%     n r_n^d \to \theta \in (0, \infty).
%     $$
%     Then $G_n$ converges locally to a homogeneous Poisson point process graph on $\mathbb{R}^d$, where each point is connected to its neighbors within distance $\theta^{1/d}$.
% \end{itemize}

\subsection{Graph convergence in the local weak sense}\label{sec:Local_Weak}
For a (possibly random) finite marked graph $(G,y)$, the empirical distribution associated with the marks is defined as
\begin{equation*}\begin{aligned}
\mu^{G,y}:=\frac{1}{|V|}\sum_{v\in V}\delta_{y_v},    
\end{aligned}\end{equation*}
which will be treated as a random element in the space $\cP(\cY)$ of probability measures on $\mathcal{Y}$. 

To study the convergence of empirical measures, we adapt the notion of marked graph convergence in probability in the local weak sense introduced in \cite{lacker2022localweakconvergencesparse}.
\begin{definition}
 Let $\mathcal{Y}$ be a Polish space. 
 % Let $y^n = (y^n_v)_{v \in G_n}$ be random $\mathcal{Y}$-valued marks on the vertices of $G_n$, and let $y = (y_v)_{v \in G}$ be random $\mathcal{Y}$-valued marks on $G$.
A sequence of marked graphs $\{(G_n,c^n,y^n)\}$ \emph{converges in probability in the local weak sense} to $(G,c,y)$ if
\begin{equation}
\lim_{n \to \infty} \frac{1}{|G_n|} \sum_{v \in G_n} \delta_{\mathcal{C}_v(G_n,c^n,y^n)}
= \cL\left(G,c,y\right)\quad \text{in probability in $\cP(\cG_*[\cY])$},
\end{equation}
    where $\mathcal{C}_v(G_n,c^n,y^n)$  denotes the connected component of the marked graph $(G_n,c^n,y^n)$ rooted at $v$.
\end{definition}
\begin{remark}\label{rem:law_root}
    If $\{(G_n,c^n,y^n)\}$ converges in probability in the local weak sense to $(G,c,y)$, then the empirical measure sequence $\left\{ \frac{1}{|G_n|} \sum_{v \in G_n} \delta_{y^n_v} \right\}$ converges in probability to $\cL(y_o)$ in $\cP(\mathcal{Y})$, where $o$ is the root of $G$.
\end{remark}

We refer the reader to \cite[Chapter 2]{Remco2024RandomG} for a detailed discussion on the local convergence of random graphs and examples.
\section{Main Results and Proofs}
\label{sec:three}
We now present the main analytical results of the paper, beginning with the existence and construction of physical solutions to the particle system defined in equations~\eqref{eq:1.1}. As discussed in Section \ref{sec:intro}, minimal fixed points of the map $\Gamma$ are natural candidates for physical solutions.

\subsection{Physical and minimal solutions}\label{sec:physical}
\begin{proposition}[Existence of minimal solution]
\label{prop:existence}
The minimal solution exists, and its $D$-component can be obtained by
\begin{equation*}\begin{aligned}
D_t:=\lim_{N\to\infty}\Gamma^{(N)}[\emptyset]_t=\bigcup_{N=0}^\infty\Gamma^{(N)}[\emptyset].  
\end{aligned}\end{equation*}
\end{proposition}
\begin{proof}
We first observe that the map $\Gamma$ is monotone: for any set-valued processes $D$ and $\tilde D$ such that $D_t \subset \tilde D_t$ for all $t \ge 0$, it holds that
\begin{equation*}
\Gamma[D]_t \subset \Gamma[\tilde D]_t, \quad \forall t \ge 0.
\end{equation*}
Applying this monotonicity iteratively starting from the empty set, the sequence $\Gamma^{(N)}[\emptyset]_t$ is increasing in $N$. Hence, we may define
\begin{equation*}
D_t := \lim_{N \to \infty} \Gamma^{(N)}[\emptyset]_t = \bigcup_{N=0}^\infty \Gamma^{(N)}[\emptyset]_t, \quad \forall t \ge 0.
\end{equation*}
We claim that $D = (D_t)_{t \ge 0}$ satisfies the fixed-point equation $\Gamma[D]_t = D_t$ for all $t \ge 0$. First, since $\Gamma^{(N)}[\emptyset]_t = \Gamma[\Gamma^{(N-1)}[\emptyset]]_t \subset \Gamma[D]_t$ for each $N$, taking the union over $N$ yields
$D_t \subset \Gamma[D]_t$.
Now we take $v\in\Gamma[D]_t$. By definition, this means
\begin{equation*}\begin{aligned}
\inf_{s\in[0,t]}\Big(x_v+Z_v(t)-\sum_{u\in N_G^-(v)}c_{uv}\bone_{\{u\in D_t\}}\Big)\leq0. 
\end{aligned}\end{equation*}
Therefore, there exists $t_0\in[0,t]$ such that 
\begin{equation*}\begin{aligned}
x_v+Z_v(t_0)-\sum_{u\in N_G^-(v)}c_{uv}\bone_{\{u\in D_{t_0}\}}\leq0.    
\end{aligned}\end{equation*}
Since $N_G^-(v)$ is a finite set and $D_{t_0} = \bigcup_{N} \Gamma^{(N)}[\emptyset]_{t_0}$, there exists a sufficiently large $N$ such that $N_G^-(v)\cap D_{t_0}=N_G^-(v)\cap\Gamma^{(N)}[\emptyset]_{t_0}$, which implies 
\begin{equation*}\begin{aligned}
x_v+Z_v(t_0)-\sum_{u\in N_G^-(v)}c_{uv}\bone_{\{u\in\Gamma^{(N)}[\emptyset]_{t_0}\}}\leq0,    
\end{aligned}\end{equation*}
that is, $v\in\Gamma^{(N+1)}[\emptyset]_{t_0}\subset D_t$. Hence, $\Gamma[D]_t\subset D_t$ for any $t\ge0$ and we conclude that $D_t = \Gamma[D]_t$ for all $t \ge 0$, and $(X,D)$ is a solution.

The minimality of $D$ follows from the monotonicity of $\Gamma$. In fact, let $(\tilde X,\tilde D)$ be any other solution. Then $\Gamma[\tilde D]=\tilde D$. Iteratively applying $\Gamma$ to $\emptyset\subset\tilde D$, we obtain $\Gamma^{(N)}[\emptyset]\subset\Gamma^{(N)}[\tilde D]=\tilde D$ for any $N\ge1$. Taking a union over $N$ yields $D \subset \tilde D$, proving minimality.
\end{proof} 

\begin{remark}\label{rem:pathwise1}
It is clear that the above construction of minimal solution is \emph{pathwise}, i.e. there exists a deterministic map $\varphi$ such that the minimal solution $(G,X,D)=\varphi(G,c,x,Z)$.
\end{remark}

However, generally speaking, minimal solutions can fail to be physical as they are not necessarily right-continuous.  We now provide the rigorous proof of the statement in  Example~\ref{ex:Minimal_Not_Physical} stated in Section~\ref{sec:intro}.

\begin{proof}[Proof of the statement in Example \ref{ex:Minimal_Not_Physical}]
We first note that
\begin{equation*}\begin{aligned}
\tau_n:=\inf\{t\ge0:X_n(t)\leq0\}\leq\inf\left\{t\ge0:\,B_t\leq-x_n\right\}=:\sigma_n.    
\end{aligned}\end{equation*}
The key observation is that $\{0,1,..,n\}\subset D_{\sigma_n}$ on the event that 
\begin{equation*}\begin{aligned}
\bigcap_{i=0}^{n-1}\left\{\sup_{s\in[0,\sigma_n]}X_i(s)\leq1\right\}\supset\bigcap_{i=0}^{n-1}\left\{\sup_{s\in[0,\sigma_n]}B_i(s)\leq1-x_i\right\}.   
\end{aligned}\end{equation*}
Conditioning on $\sigma_n$ and using the reflection principle and the density of the Brownian first hitting time (see \cite[Theorem 3.7.1]{shreve2004stochastic}), we compute:
\begin{equation*}\begin{aligned}
\PP\left[\bigcap_{i=0}^{n-1}\left\{\sup_{s\in[0,\sigma_n]}B_i(s)\leq1-x_i\right\}\right]&=\int_0^\infty\prod_{i=0}^{n-1}\PP\left[\sup_{s\in[0,t]}B_i(s)\leq1-x_i\right]\,\PP[\sigma_n\in\,\mathrm{d}t]\\
&\ge\int_0^\infty\left(\PP\left[\sup_{s\in[0,t]}B_i(s)\leq\frac12\right]\right)^n\,\PP[\sigma_n\in\,\mathrm{d}t]\\
&=\int_0^\infty\left(1-2\Phi\left(-\frac{1}{2\sqrt{t}}\right)\right)^n\,\cdot\frac{x_n}{t\sqrt{2\pi t}}e^{-\frac{x_n^2}{2t}}\,\mathrm{d}t\\
(\text{change of variable: }t\mapsto x_n^2t)\quad&=\int_0^\infty\left(1-2\Phi\left(-\frac{1}{2x_n\sqrt{t}}\right)\right)^n\,\cdot\frac{1}{t\sqrt{2\pi t}}e^{-\frac{1}{2t}}\,\mathrm{d}t,
\end{aligned}\end{equation*}
where $\Phi$ denotes the cumulative distribution function of the standard Gaussian distribution. By the dominated convergence theorem and the Gaussian tail bound $\Phi(-x)=1-\Phi(x)\le e^{-x^2/2}$ for any $x>0$, the above integral converges to $1$ as $n \to \infty$.
Moreover, for any $\varepsilon>0$,
\begin{equation*}\begin{aligned}
\sum_{n=1}^\infty\PP[\sigma_n>\varepsilon]=\sum_{n=1}^\infty\PP[\inf_{s\in[0,\varepsilon]}B_n(s)>-x_n]=\sum_{n=1}^\infty\PP[|B(\varepsilon)|\leq x_n]\leq\sum_{n=1}^\infty\frac{2}{\sqrt{2\pi\varepsilon}}x_n<\infty.    
\end{aligned}\end{equation*}
By the Borel–Cantelli lemma, we conclude that $\sigma_n \to 0$ almost surely. Therefore, $D_{0+}=V=\NN$ with probability $1$ for any solution $(X,D)$. Since $x_n>0$ for all $n\ge0$, any physical solution $(X,D)$ must satisfy $D_0=\emptyset$ by the definition of the physicality condition. As a result, the right continuity will be violated and hence there is no physical solution to equations \eqref{eq:1.1} with the given initial conditions.
\end{proof}

\subsection{Fragility, robustness and default cascades}\label{sec:robustness}
Recall the definition of fragility and robustness mentioned in Section \ref{sec:intro}.
\begin{definition}
A vertex $v\in V$ is said to be \emph{fragile} at time $t$ subject to the configuration $(G,c,x,Z)$, if 
\begin{equation*}\begin{aligned}
v\in F_t:=\Big\{v\in V:\,x_v+Z_v(t-)\ge0\quad\text{and}\quad x_v+Z_v(t)\leq \sum_{u\in N_G^-(v)}c_{uv}\Big\}. 
\end{aligned}\end{equation*}
\end{definition}
\begin{remark}
Let $(X,D)$ be a solution to the system \eqref{eq:1.1}. Then it is straightforward to see that $v\in F_{\tau_v}$ provided that $\tau_v<\infty$, where
\begin{equation*}
\tau_v:=\inf\{t\ge0:\,X_v(t)\leq0\}=\inf\{t\ge0:\,v\in D_t\}.    
\end{equation*}
\end{remark}

% \begin{remark}
%     There is a broad class of initial conditions that satisfy the above assumption. For example, suppose that $w_v \leq r$ for every $v \in V$, and that the initial states $x_v$ are independently distributed according to a uniform distribution on $ [x_1, x_2]$, where $0 < r < x_1 < x_2$. Then, by standard Gaussian tail estimates, the probability that $x_v$ approaches $w_v$ within a short time interval  $[0,\delta]$ is controlled by
% $\mathbb{P}(x_v + B_\delta \leq r) \sim \exp\left( -\frac{(x_1 - r)^2}{2\epsilon} \right).$
% Then, 
% $$0 < \delta < \frac{(x_1 - r)^2}{2|\log p_c|}$$
% is  sufficient to ensure that the assumption holds true.
% \end{remark}

In the next lemma, we outline a few sufficient conditions for $(G,c,x,Z)$ to drive $\delta$-robust systems.
\begin{lemma}\label{lem:robust1}
 The configuration $(G,c,x,Z)$ drives $\delta$-robust systems for some $\delta>0$, provided that at least one of the following conditions holds:
 \begin{enumerate}
    \item $G$ is finite.
    \item (i.i.d. Lévy noises) $G$ is countably infinite, locally finite and
        \begin{itemize}
        \item[a.] There exists $w>0$, such that $w_v\leq w$ for every $v\in V$.
        \item[b.] $(Z_v)_{v\in V}$ are i.i.d. Lévy processes and $(x_v)_{v\in V}\stackrel{\mathrm{i.i.d.}}{\thicksim}\mu$.
        \item[c.] There exists $\eta>0$, such that
    \begin{equation}\label{eq:percolation1}
        p_\eta:=\sup_{t\ge0} \PP[x_v+Z_v(t-)\in (-\eta,w+\eta)]<p_c, 
        \end{equation}
        where $p_c$ is the threshold for site percolation of the graph $G$  \cite[Chapter 1]{grimmett1999percolation}.
        \end{itemize}
    \item (common noise) $G$ is countably infinite, locally finite and
        \begin{itemize}
        \item [a.] There exists $w>0$, such that $w_v\leq w$ for every $v\in V$.
        \item[b.] There exist i.i.d. Lévy processes $(\tilde Z_v)_{v\in V}$ and another càdlàg process $Z^0$ that is independent of $(\tilde Z_v)_{v\in V}$ such that $Z_v(t)=\tilde Z_v(t)+Z^0(t)$, and
        $(x_v)_{v\in V}\stackrel{\mathrm{i.i.d.}}{\thicksim}\mu$.
        \item[c.] With probability $1$, there exists $\eta>0$, such that
        \begin{equation}\label{eq:percolation2}
        \sup_{t\ge0} \PP[x_v+\tilde Z_v(t-)+Z^0(t)\in (-\eta,w+\eta)\,|\,Z^0]<p_c.
        \end{equation}
        \end{itemize}
    \item (scaled symmetry) There exists a configuration $(\tilde G,\tilde c,\tilde x,\tilde Z)$ satisfying one of conditions 1-3, and there exist (potentially random) positive variables $(\lambda_v)_{v\in V}$ such that 
    \begin{equation*}\begin{aligned}
    G=\tilde G,\quad x_v=\lambda_v\tilde x_v,\quad c_{uv}=\lambda_v\tilde c_{uv},\quad Z_v=\lambda_v\tilde Z_v.    
    \end{aligned}\end{equation*}
    \end{enumerate}
\end{lemma}
\begin{remark}\label{rem:robust1}
In practice, financial systems exhibit significant size disparities. The scaling factors $(\lambda_v)_{v\in V}$ in condition 4 capture these size differences, allowing larger institutions to have proportionally scaled initial capital, exposures, and volatilities relative to smaller counterparts. The idea is that while institutions differ in scale, the relative risk profile—captured by the ratios of capitals to potential losses—remains consistent with a base symmetric model, making the analysis tractable while preserving realistic heterogeneity in institutional sizes and interconnections.
\end{remark}
\begin{proof}
For condition 1, note that the set of fragile vertices will be automatically finite if the whole graph $G$ is finite.

We now focus on condition 2. Fix $\delta > 0$ to be chosen later. For each $v \in V$ and $t\ge0$, define the event
\begin{equation*}\begin{aligned}
A^{(v)}_{t,2\delta} :=\left\{ x_v + Z_v(s-)<0\,\text{ or }\,x_v+Z_v(s)>w,\ \forall s \in [t, t +2\delta] \right\}\subset\left\{ v\notin F_s, \ \forall s \in [t, t + 2\delta] \right\}.
\end{aligned}\end{equation*}
To ease the notation, we omit the index $v$ when there is no ambiguity. Notice that for $s>t$, $x+Z(s)=(x+Z(t-))+\tilde Z(s-t)$ where $\tilde Z(r):= Z(t+r)-Z(t-)$ is another Lévy process with the same distribution as $Z$, and is independent of $x+Z(t-)$. Denote by $\mu_{t-}$ the law of $x+Z(t-)$. We rewrite $A_{t,2\delta}$ as
\begin{equation*}
A_{t,2\delta}=\{\forall r\in [0,2\delta],\, x+Z(t-)+\tilde Z(r-)<0\,\text{ or }x+Z(t-)+\tilde Z(r)>w\}.
\end{equation*}
From equation \eqref{eq:percolation1}, there exist $\eps>0$ and $\eta>0$, such that $p_\eta<p_c-\eps$. For $z\geq w+\eta$, 
\begin{equation*}\begin{aligned}
q^+(z,2\delta):=\PP[\forall r\in [0,2\delta], z+\tilde Z(r)>w]\ge\PP\Big[\inf_{r\in [0,2\delta]}\tilde Z(r)>-\eta\Big].
\end{aligned}\end{equation*}
Similarly, for $z\leq -\eta$, 
\begin{equation*}\begin{aligned}
q^-(z,2\delta):=\PP[\forall r\in [0,2\delta], z+\tilde Z(r)<0]\ge\PP\Big[\sup_{r\in [0,2\delta]}\tilde Z(r)<\eta\Big].
\end{aligned}\end{equation*}
Let $q(z,2\delta):=q^+(z,2\delta)\bone_{\{z\geq w+\eta\}}+q^-(z,2\delta)\bone_{\{z\leq -\eta\}}$. Then,
\begin{equation*}\begin{aligned}
\PP[A_{t,2\delta}]&=\EE\left[\PP[A_{t,2\delta}|\,x+Z(t-)]\right]\ge\EE[q(x+Z(t-),2\delta)]\\
    &=\int_{-\infty}^{-\eta} q^-(z,2\delta)\mu_{t-}(dz)+\int_{w+\eta}^{\infty}q^+(z,2\delta)\mu_{t-}(dz)\\
    &\ge\PP\Big[\sup_{r\in [0,2\delta]}|\tilde Z(r)|<\eta\Big]\left(1-\PP[x+Z(t-)\in (-\eta,w+\eta)]\right).
\end{aligned}\end{equation*}
From equation \eqref{eq:percolation1}, $p_\eta=\sup_{t>0} \PP[x+Z(t-)\in (-\eta,w+\eta)]\leq p_c-\eps$. In addition, as $\tilde Z$ is a Lévy process, we can choose $\delta$ sufficiently small such that
\begin{equation*}\begin{aligned}
\PP\Big[\sup_{r\in [0,2\delta]}|\tilde Z(r)|<\eta\Big]\ge\frac{1-p_c}{1-p_c+\varepsilon}.
\end{aligned}\end{equation*}
Then $\PP[A_{t,2\delta}]\geq 1-p_c$ for every $t>0$. Since the $(x_v, Z_v)$ are i.i.d., the events $A^{(v)}_{t,2\delta}$ are independent across $v\in V$. Therefore, by the percolation theory (see \cite[Section 1.4-1.6]{grimmett1999percolation}),
\begin{equation*}\begin{aligned}
\mathbb{P}\Bigg[\bigcup_{s\in[t,t+2\delta]}F_s \text{ contains no infinite weakly connected component}\Bigg] = 1,\quad\forall t\ge0.
\end{aligned}\end{equation*}
By taking a countable intersection over $t=0,\delta,2\delta,...$, we obtain
\begin{equation*}\begin{aligned}
\mathbb{P}\Bigg[\bigcap_{n=0}^\infty\Big\{\bigcup_{s\in[n\delta,(n+2)\delta]}F_s \text{ contains no infinite weakly connected component}\Big\}\Bigg] = 1,
\end{aligned}\end{equation*}
which implies $\delta$-robustness as any subinterval of $[0,\infty)$ of length $\delta$ must be contained in $[n\delta,(n+2)\delta]$ for some $n\ge0$.

For condition 3, the proof can be done by first conditioning on $Z$ and then repeating the previous argument for condition 2. For condition 4, it suffices to notice that the scaled system has the same fragile set as the original system, for any $t\ge0$.
\end{proof}

We then turn to show some technical results concerning minimal solutions for which the configuration $(G,c,x,Z)$ drives robust systems:
\begin{proposition}\label{prop:MinSolution_RightCont}
Let the configuration $(G,c,x,Z)$ drive robust systems and let $(X,D)$ be the minimal solution. Then
\begin{equation*}\begin{aligned}
\PP[D_t=D_{t+},\quad\forall t\ge0]=1.    
\end{aligned}\end{equation*}
\end{proposition}
\begin{proof}
Suppose there exists $t\ge0$ and a vertex $v\in D_{t+}\setminus D_t\neq\emptyset$. We observe that, necessarily,
\begin{equation*}\begin{aligned}
X_v(t)&=x_v+Z_v(t)-\sum_{u\in N_G^-(v)}c_{uv}\bone_{\{u\in D_t\}}>0,\\ X_v(t+)&=x_v+Z_v(t)-\sum_{u\in N_G^-(v)}c_{uv}\bone_{\{u\in D_{t+}\}}\leq0.   
\end{aligned}\end{equation*}
This implies $v\in D_{t+}\setminus D_t\subset F_t$. By the robustness assumption, $F_t$ admits a decomposition
\begin{equation*}\begin{aligned}
F_t=\bigcup_{n}F_t^{n},   
\end{aligned}\end{equation*}
where each $F_t^n$ is a finite weakly connected component. Then there exists $n$ such that $v\in F_t^n$. For each $w\in F_t\setminus D_t$, consider the hitting time
\begin{equation*}\begin{aligned}
\sigma_w:=\inf\Big\{s\ge t:\,x_w+Z_w(s)-\sum_{u\in N_G^-(v)}c_{uw}\bone_{\{u\in D_t\}}\leq0\Big\}>t.    
\end{aligned}\end{equation*}
By the finiteness of $F_t^n$, we know that $\min_{w\in F_t^n\setminus D_t}\sigma_w>t$. In addition, we observe that 
\begin{equation*}\begin{aligned}
x_u+Z_u(t)-\sum_{w\in N_G^-(u)}c_{wu}>0    
\end{aligned}\end{equation*}
for any $u\in F_t^c\setminus D_t$. From the definition of outer boundary of $F^n_t$,
\begin{equation*}\begin{aligned}
\partial^{\mathrm{out}}F_t^n=\{u\in(F_t^n)^c:\,\exists v\in F_t^n\,\,\text{s.t. }(u,v)\in E\,\,\text{or }(v,u)\in E\},    
\end{aligned}\end{equation*}
which is finite and has empty intersection with $F_t$ due to the maximal connectedness of $F_t^n$. In other words, $\partial^{\mathrm{out}}F_t^n\subset F_t^c$. For each $u\in\partial^{\mathrm{out}}F_t^n\setminus D_t$, consider the hitting time
\begin{equation*}\begin{aligned}
\sigma_u:=\inf\Big\{s\ge t:\,x_u+Z_u(s)-\sum_{w\in N_G^-(u)}c_{wu}\leq0\Big\}>t.    
\end{aligned}\end{equation*}
By the finiteness of $\partial^{\mathrm{out}}F_t^n$, we know that $\min_{u\in\partial^{\mathrm{out}}F_t^n\setminus D_t}\sigma_u>t$. Now, define
\begin{equation*}\begin{aligned}
\sigma:=(\min_{w\in F_t^n\setminus D_t}\sigma_w)\wedge(\min_{u\in\partial^{\mathrm{out}}F_t^n\setminus D_t}\sigma_u)>t.    
\end{aligned}\end{equation*}
The claim is that for $(F_t^n\setminus D_t)\cap D_s=\emptyset$ for all $s\in(t,\sigma)$. To see this, we first point out that $w\in F_t^{n}\setminus D_t$ implies
\begin{equation*}\begin{aligned}
&\inf_{s\in[0,t]}(x_w+Z_w(s))\ge\inf_{s\in[0,t]}X_w(s)>0\quad\text{and}\\& x_w+Z_w(s)\ge x_w+Z_w(s)-\sum_{u\in N_G^-(w)}c_{uw}\bone_{\{u\in D_t\}}>0,\quad\forall s\in(t,\sigma),    
\end{aligned}\end{equation*}
which implies $(F_t^n\setminus D_t)\cap\Gamma^{(0)}[\emptyset]_s=\emptyset$ for all $s\in(t,\sigma)$. Now, suppose $(F_t^n\setminus D_t)\cap\Gamma^{(N)}[\emptyset]_s=\emptyset$ for all $s\in(t,\sigma)$. We fix any $w\in F_t^n\setminus D_t$ and any $s\in(t,\sigma)$. The condition $(F_t^n\setminus D_t)\cap\Gamma^{(N)}[\emptyset]_s=\emptyset$ implies that, for any $u\in N_G^-(v)\cap\Gamma^{(N)}[\emptyset]_s$, either $u\in D_t$ or $u\in\partial(F_t^n)^c\setminus D_t$. In the former case, $\bone_{\{u\in\Gamma^{(N)}[\emptyset]_s\}}=1=\bone_{\{u\in D_t\}}$. In the latter case, $\bone_{\{u\in\Gamma^{(N)}[\emptyset]_s\}}=\bone_{\{u\in D_s\}}=0$. Combining the above, we get
\begin{equation*}\begin{aligned}
x_w+Z_w(s)-\sum_{u\in N_G^-(w)}c_{uw}\bone_{\{u\in \Gamma^{(N)}[\emptyset]_s\}}\ge x_w+Z_w(s)-\sum_{u\in N_G^-(w)}c_{uw}\bone_{\{u\in D_t\}}>0    
\end{aligned}\end{equation*}
for all $s\in(t,\sigma)$. In addition, for any $w\in F_t^n\setminus D_t$, it holds that
\begin{equation*}\begin{aligned}
&\inf_{s\in[0,t]}\Big(x_w+Z_w(s)-\sum_{u\in N_G^-(w)}c_{uw}\bone_{\{u\in\Gamma^{(N)}[\emptyset]_s\}}\Big)\\
&\ge\inf_{s\in[0,t]}\Big(x_w+Z_w(s)-\sum_{u\in N_G^-(w)}c_{uw}\bone_{\{u\in D_s\}}\Big)>0.    
\end{aligned}\end{equation*}
The above two estimates combined together imply that $w\notin\Gamma^{(N+1)}[\emptyset]_s$. By the arbitrariness of $w\in F_t^n\setminus D_t$, we obtain $(F_t^n\setminus D_t)\cap\Gamma^{(N+1)}[\emptyset]_s=\emptyset$ for all $s\in(t,\sigma)$. Then it can be shown inductively that $(F_t^n\setminus D_t)\cap\Gamma^{(N)}[\emptyset]_s=\emptyset$ for any $N\ge0$ and therefore $(F_t^n\setminus D_t)\cap D_s=\emptyset$ for all $s\in(t,\sigma)$, which further implies that $(F_t^n\setminus D_t)\cap D_{t+}=\emptyset$. As the choice of $n$ is arbitrary and $D_{t+}\setminus D_t\subset F_t$, it necessarily holds that $D_{t+}\setminus D_t=\emptyset$. 
\end{proof}

\begin{proposition}\label{prop:MinimalIsPhysical}
Let the configuration $(G,c,x,Z)$ drive robust systems. Then the minimal solution is physical.    
\end{proposition}
\begin{proof}
Let $(X,D)$ be the minimal solution. It has been shown in Prop \ref{prop:MinSolution_RightCont} that $t\mapsto D_t$ is almost surely right-continuous. It remains to verify that the jump size at each discontinuity time satisfies Definition~\ref{def:physical}. Fix any time $t \ge 0$ such that $D_{t-} \subsetneq D_t$. First, it is easy to see that $D_t^{(0)}\subset D_t$. Assuming that $D_t^{(N)}\subset D_t$, we get 
\begin{equation*}\begin{aligned}
D_t^{(N+1)}&:=D_t^{(N)}\cup\Big\{v\in V:\,x_v+Z_v(t)-\sum_{u\in N_G^-(v)}c_{uv}\bone_{\{u\in D_t^{(N)}\}}\leq0\Big\}\\
&\subset D_t^{(N)}\cup\Big\{v\in V:\,x_v+Z_v(t)-\sum_{u\in N_G^-(v)}c_{uv}\bone_{\{u\in D_t\}}\leq0\Big\}\subset D_t.
\end{aligned}\end{equation*}
As a result, it follows from an induction argument that $D_t^{(N)}\subset D_t$ for any $N\ge1$ and thus $D_t^{(\infty)}\subset D_t$. Conversely, as $\emptyset\subset D_t^{(0)}$, we can iteratively show that $\Gamma^{(N)}[\emptyset]_t\subset D_t^{(N)}$ for any $N\ge1$ and thus $D_t\subset D_t^{(\infty)}$ by taking $N\to\infty$. Combining both directions, we conclude that  $D_t=D_t^{(\infty)}$.
\end{proof}
\begin{corollary}
Let $(X,D)$ be the minimal solution. With probability $1$, for any $t\ge0$,
\begin{equation*}\begin{aligned}
x_v+Z_v(t)-\sum_{u\in N_G^-(v)}c_{uv}\bone_{\{u\in D_{t-}\}}>0,\quad\forall v\notin D_{t-}\quad\text{implies }\quad D_t=D_{t-}.
\end{aligned}\end{equation*}
\end{corollary}
\begin{proof}
With the given assumption, we have $D_t^{(0)}=D_{t-}$. Then it can be inductively shown that $D_t^{(N)}=D_{t-}$ for any $N\ge1$. Therefore, $D_t=D_{t-}$ by the physicality condition.
\end{proof}

It has been shown in Proposition \ref{prop:MinimalIsPhysical} that the minimal solution is physical, under the robustness assumption. To complete the picture, we investigate whether physical solution must be the minimal one. As it turns out, the slightly stronger $\delta$-robustness assumption would yield an affirmative answer.
\begin{definition}
Consider a physical solution $(X,D)$. For any $v\in V$, define
\begin{equation*}\begin{aligned}
\tau_v:=\inf\{t\ge0:X_v(t)\leq0\},\quad k_v:=\min\{N\ge0:v\in D_{\tau_v}^{(N)}\}.    
\end{aligned}\end{equation*}
Here, we adopt the convention that $\inf\emptyset=\infty$. We say that $u$ defaults before $v$, denoted by $u\prec v$, if $(\tau_u,k_u)<(\tau_v,k_v)$ in the lexicographical order. That is,
\begin{equation*}\begin{aligned}
\tau_u<\tau_v\quad\text{or}\quad(\tau_u=\tau_v\quad\text{and}\quad k_u<k_v).  
\end{aligned}\end{equation*}
\end{definition}

\begin{proposition}
Let the configuration $(G,c,x,Z)$ drive $\delta$-robust systems, for some $\delta>0$. Let $(X,D)$ be a physical solution. Then it is the minimal solution.
\end{proposition}
\begin{proof}
Let $(\underline X,\underline D)$ be the minimal solution, and let $t_0:=\inf\{t\ge0,D_t\neq\underline D_t\}$. Suppose that, for contradiction, $t_0<\infty$. Then $D_s=\underline D_s$ and also $X_v(s)=\underline X_v(s)$ for any $s\in[0,t_0)$ and any $v\in V$. By continuity of $X_v$
this implies that $D_{t_0}^{(0)}=\underline D_{t_0}^{(0)}$. Iterating the physical condition then gives $D_{t_0}=\underline D_{t_0}$. Now, it follows from the definition of $t_0$ that, for any $\delta>0$, there must exist $t_\delta\in(t_0,t_0+\delta)$ and  $v_\delta^{(0)}\in D_{t_\delta}\setminus\underline D_{t_\delta}$. Let $(\tau,k):=(\tau_{v_\delta^{(0)}},k_{v_\delta^{(0)}})$, then we must have $v_\delta^{(0)}\in D_\tau^{(k)}\setminus\underline D_\tau$, which implies
\begin{equation}\begin{aligned}\label{eq:UniqueMinEq1}
&x_{v_\delta^{(0)}}+Z_{v_\delta^{(0)}}(\tau)-\sum_{u\in N_G^-({v_\delta^{(0)}})}c_{u{v_\delta^{(0)}}}\bone_{\{u\in D_{\tau}^{(k-1)}\}}\leq0\\
\quad\text{while}\quad&x_{v_\delta^{(0)}}+Z_{v_\delta^{(0)}}(\tau)-\sum_{u\in N_G^-({v_\delta^{(0)}})}c_{u{v_\delta^{(0)}}}\bone_{\{u\in\underline D_{\tau}\}}>0.
\end{aligned}\end{equation}
Hence there must exist $v_\delta^{(1)}\in N_G^-(v_\delta^{(0)})\cap(D_\tau^{(k-1)}\setminus\underline D_\tau$). In other words, $v_\delta^{(1)}$ is an element of $N_G^-(v_\delta^{(0)})$ that defaults before $v_\delta^{(0)}$ in the system $(X,D)$. Moreover, Equations \eqref{eq:UniqueMinEq1} imply that $v_\delta^{(0)}\in F_\tau\subset\bigcup_{t\in[t_0,t_0+\delta]}F_t$. Iterating this argument, we get a sequence $(v_\delta^{(n)})_{n\ge0}\subset D_{t_\delta}\setminus\underline D_{t_\delta}$ such that $v_\delta^{(n+1)}\in N_G^-(v_\delta^{(n)})$ and that $v_\delta^{(n+1)}$ defaults before $v_\delta^{(n)}$ in the system $(X,D)$. The existence of such a sequence implies that $\bigcup_{t\in[t_0,t_0+\delta]}F_t$ has an infinitely large weakly connected component, which is a contradiction to the $\delta$-robustness assumption. As a result, we must have $t_0=\infty$ and thus $D_t=\underline D_t$ for any $t\ge0$. That is, $(X,D)$ is the minimal solution.
\end{proof}
\begin{corollary}
\label{cor:uniqueness}
The physical solution is unique.    
\end{corollary}

\subsection{Default trees and recovery of locality}\label{sec:Trees}
Before discussing the proof of Theorem \ref{thm-main-2}, we need a more accurate description of the interaction among the particles that are potentially distant away from each other, which is of interest on its own. This is analyzed via the growth of \textit{default trees}, defined in the following.

A tree is a directed acyclic graph in which each vertex is the out-neighbor of exactly one vertex, except for the root vertex which is not the out-neighbor of any edges. The out-neighbors of a vertex is called its \emph{children}. Vertices that have no out-neighbors in the tree are called the \emph{leaves} of the tree.

\begin{definition}
Let $(X,D)$ be a physical solution associated with a problem configuration $(G,c,x,Z)$. For any vertex $v_0\in V$ such that $\tau_{v_0}<\infty$, the \emph{default tree} $\cT(G,v_0)$ rooted at $v_0$ is a subgraph of $G$ defined recursively as follows:
\begin{enumerate}
    \item The root is $v_0$.
    \item For any node $v$,  its children are the vertices $u \in N_G^-(v)$ such that $u \prec v$.
\end{enumerate}
If $v_0$ is such that $\tau_{v_0}=\infty$, we define $\cT(G,v_0)$ to be the empty graph.
\end{definition}

% \begin{lemma}
% With probability $1$, for any $v_0\in V$ such that $\tau_{v_0}<\infty$, it holds that
% \begin{enumerate}
%     \item $\cT^\leftarrow(G,v_0)$ is finite. 
%     \item $\cT^\rightarrow(G,v_0)\cap\underline D_T$ is finite, for any $T\in[0,\infty)$.
% \end{enumerate} 
% \end{lemma}
% The proof of the above Lemma is to be skipped for now, as we will prove a stronger result in Lemma \ref{lma:finitetree}.

% We first state a useful lemma from random graph theory. The proof can be found in \cite{grimmett1999percolation}.

% \begin{lemma}[Exponential decay]
% \label{lma:expdecay}
%     Given our main robustness assumption~\ref{asm:main}, there exists $M>1$, $c_1,c_2>0$, such that for every time $t>0$, every $v\in G$ and every $N>0$,
%     $$
%     \PP\left[|\cT(G,v)\cap D_{[t,t+\delta)}|>N\right]\leq c_1 e^{-c_2 N},\quad \text{and}\quad \E\left[|\cT(G,v)\cap D_{[t,t+\delta)}|\right]\leq M.
%     $$
% \end{lemma}

\begin{lemma}
\label{lma:finitetree}
Let the configuration $(G,c,x,Z)$ drive $\delta$-robust systems for some $\delta>0$, and let a sequence of configurations $(G_n, c^n, x^n, Z^n)_{n \ge 1}$  converges  to $(G, c, x, Z)$ locally in $\cG_*[\mathbb{R} \times \mathcal{D}]$ 
 almost surely as $n\to\infty$. Then with probability one, for any fixed $v_0 \in V$ satisfying either: 
\begin{itemize}
    \item $\tau_{v_0}<\infty$, or
    \item $\tau_{v_0}=\infty$ but $\lim_{n\to\infty}\tau_{v_0}^n$ exists and is finite,
\end{itemize}
there exists a finite subset $V_0$ of $V$ such that:
\begin{enumerate}
    \item $v_0\in V_0$ and $\cT(G,v_0)\subset V_0$.
    \item For all sufficiently large $n$, $\cT(G_n,v_0)\subset V_0$.
\end{enumerate}
\end{lemma}

\begin{proof}[Proof of Lemma \ref{lma:finitetree}]
\noindent\textbf{Step 1.} In the case that $\tau_{v_0}<\infty$, we take $M$ to be the unique integer such that $M\delta\leq\tau_{v_0}<(M+1)\delta$. In the case that $\tau_{v_0}=\infty$ but $\lim_{n\to\infty}\tau_{v_0}^n<\infty$, we take $M$ to be the unique integer such that $M\delta\leq\tau:=\lim_{n\to\infty}\tau_{v_0}^n<(M+1)\delta$. Then for each $m\in\{0,1,2,...,M\}$, the set $F_{[m\delta,(m+1)\delta]}$ admits the following decomposition into its weakly connected components
\begin{equation*}\begin{aligned}
F_{[m\delta,(m+1)\delta]}=\bigcup_{l\in\NN} F_{[m\delta,(m+1)\delta]}^l    
\end{aligned}\end{equation*}
such that each $F_{[m\delta,(m+1)\delta]}^l$ is finite. In the case that $M\delta\leq\tau_{v_0}<(M+1)\delta$, we have $v_0\in F_{\tau_{v_0}}\subset F_{[M\delta,(M+1)\delta]}$. In the case that $M\delta\leq\tau=\lim_{n\to\infty}\tau_{v_0}^n<(M+1)\delta$, we see that
\begin{equation*}\begin{aligned}
x_{v_0}^n+Z_{v_0}^n(\tau_{v_0}^n-)\ge0\quad\text{and}\quad x_{v_0}^n+Z_{v_0}^n(\tau_{v_0}^n)\leq\sum_{u\in N_G^-(v_0)}c_{uv_0}^n.    
\end{aligned}\end{equation*}
Taking $n\to\infty$, applying Lemma \ref{lem:D_Conv_Pointwise2} and making use of the assumption that $Z_{v_0}$ does not jump upward (which implies $Z_{v_0}(\tau)\leq Z_{v_0}(\tau-)$), we obtain
\begin{equation*}\begin{aligned}
x_{v_0}+Z_{v_0}(\tau-)\ge0\quad\text{and}\quad x_{v_0}+Z_{v_0}(\tau)\leq\sum_{u\in N_G^-(v_0)}c_{uv_0}.    
\end{aligned}\end{equation*}
and thus $v_0\in F_\tau\subset F_{[M\delta,(M+1)\delta]}$.
Therefore, in both cases, we can take $l_M$ to be such that $v_0\in F_{[M\delta,(M+1)\delta]}^{l_M}$. Define $\mathcal{L}_M := \{l_M\}$, and recursively for $m = M-1, \dots, 0$ define
\begin{equation*}\begin{aligned}
\cL_m:=\{l\in\NN:\,\exists m'\in\{m+1,...,M\}\,\exists l'\in\cL_{m'},F_{[m\delta,(m+1)\delta]}^l\cap B_G(F_{[m'\delta,(m'+1)\delta]}^{l'},1)\neq\emptyset\}. 
\end{aligned}\end{equation*}
In particular, if $M\ge1$ and either $\tau_{v_0}=M\delta$ or $\tau=\lim_{n\to\infty}\tau_{v_0}^n=M\delta$, there also exists $l\in\cL_{M-1}$ such that $v_0\in F_{[(M-1)\delta,M\delta]}^{l}$. Note that each $\cL_m$ must be finite when the graph $G$ is locally finite. Finally, we define 
\begin{equation*}\begin{aligned}
V_0=\bigcup_{m=0}^M\bigcup_{l\in\cL_m}F_{[m\delta,(m+1)\delta]}^l,    
\end{aligned}\end{equation*}
which is then a finite set.\\

\noindent\textbf{Step 2.} In this step, we verify that $\cT(G,v_0)\subset V_0$. Assume $\tau_{v_0} < \infty$ (otherwise the tree is empty). We proceed by top-down induction on $\mathcal{T}(G, v_0)$. For the base case, it is clear that $v_0\in F_{[M\delta,(M+1)\delta]}^{l_M}\subset V_0$. For the inductive step, let's assume that $v,v'$ are nodes such that $\tau_v\in[m\delta,(m+1)\delta)$, $v\in F_{[m\delta,(m+1)\delta]}^{l_v}$ where $l_v\in\cL_m$, $\tau_{v'}\in[m'\delta,(m'+1)\delta)$, and that $v'$ is a child of $v$ in the tree $\cT(G_n,v_0)$. If $m'=m$, then $v'\in F_{[m\delta,(m+1)\delta]}$, and since $v'\in N_G(v)$, it must hold that $v'\in F_{[m\delta,(m+1)\delta]}^{l_v}$ by the maximal connectedness of the latter. If $m'<m$, then there exists $l$ such that $v'\in F_{[m'\delta,(m'+1)\delta]}^l\cap B_G(F_{[m\delta,(m+1)\delta]}^{l_v},1)$, which implies $l\in\cL_{m'}$ by the definition of $\cL_{m'}$. The induction argument is now complete, which implies that $\cT(G,v_0)\subset V_0$. In particular, $\cT(G,v_0)$ is a finite set.\\

\noindent\textbf{Step 3.} As an intermediate step, we prove that $\tau_{v_0}\ge\limsup_{n\to\infty}\tau_{v_0}^n$. Since the desired inequality is trivial if $\tau_{v_0}=\infty$, we focus on the case in which $\tau_{v_0}<\infty$. As $\cT(G,v_0)$ is a finite tree, we can use a bottom-up induction to show that $\tau_v\ge\limsup_{n\to\infty}\tau_v^n$ for all $v\in\cT(G,v_0)$. For the base case, suppose $v$ is a leaf node of $\cT(G,v_0)$. Then necessarily
$X_v(\tau_v)=x_v+Z_v(\tau_v)\leq0$. With probability $1$, for any $\Delta>0$, there exists $t_\Delta\in(\tau_v,\tau_v+\Delta)$ ,which can be taken to be a continuity point of $Z_v(\cdot)$, such that $x_v+Z_v(t_\Delta)<0$. For all sufficiently large $n$, it must hold that $x_v^n+Z_v^n(t_\Delta)<0$ and thus $\limsup_{n\to\infty}\tau_v^n\leq\tau_v+\Delta$. Letting $\Delta\downarrow0$ gives $\tau_v\ge\limsup_{n\to\infty}\tau_v^n$. For the inductive step, we assume that $v$ is a node such that any of its children $u$ in $\cT(G,v_0)$ satisfies $\tau_u\ge\limsup_{n\to\infty}\tau_u^n$. Similarly, as
\begin{equation*}\begin{aligned}
X_v(\tau_v)=x_v+Z_v(\tau_v)-\sum_{u\in N_G^-(v)}c_{uv}\bone_{\{u\in D_{\tau_v}\}}\leq0,  
\end{aligned}\end{equation*}
with probability $1$, for any $\Delta>0$, there exists $t_\Delta\in(\tau_v,\tau_v+\Delta)$, that can be taken to be a continuity point of $Z_v(\cdot)$, such that
\begin{equation*}\begin{aligned}
x_v+Z_v(t_\Delta)-\sum_{u\in N_G^-(v)}c_{uv}\bone_{\{u\in D_{\tau_v}\}}<0.
\end{aligned}\end{equation*}
For all sufficiently large $n$, it holds that $\tau_u^n\leq t_\Delta$ for all children $u$ of $v$ in $\cT(G,v_0)$ and 
\begin{equation*}\begin{aligned}
x_v^n+Z_v^n(t_\Delta)-\sum_{u\in N_G^-(v)}c_{uv}^n\bone_{\{u\in D_{t_\Delta}^n\}}<0,
\end{aligned}\end{equation*}
which, in turn, implies $\limsup_{n\to\infty}\tau_v^n\leq t_\Delta\leq\tau_v+\Delta$. Letting $\Delta\downarrow0$ gives $\tau_v\ge\limsup_{n\to\infty}\tau_v^n$. The induction argument is thus completed.\\

\noindent\textbf{Step 4.} Now, we are going to verify that $V_0$ satisfies the desired properties. First, as $V_0$ is a finite set, there exists a sufficiently large $k\in\NN$ such that $V_0\subset B_G(o,k)$, where $o$ is the root of $G$. As $G_n$ converges locally to $G$ in $\cG_*$, there exists $N_0\in\NN$ such that $B_{G_n}(o_n,k+2)=B_G(o,k+2)$ for any $n\ge N_0$. Second, we note that for any $v\in B_G(V_0,1)$ such that $v\notin F_{[m\delta,(m+1)\delta]}$, it holds that
\begin{equation*}\begin{aligned}
\sup_{s\in[m\delta,(m+1)\delta]}(x_v+Z_v(s))<0\quad\text{or}\quad\inf_{s\in[m\delta,(m+1)\delta]}(x_v+Z_v(s))>\sum_{u\in N_G^-(v)}c_{uv}.
\end{aligned}\end{equation*}
As $Z_v(\cdot)$ is càdlàg, there exists a $\delta_v\in(0,\delta)$ such that 
\begin{equation*}\begin{aligned}
\sup_{s\in[m\delta-\delta_v,(m+1)\delta+\delta_v]}(x_v+Z_v(s))<0\quad\text{or}\quad\inf_{s\in[m\delta-\delta_v,(m+1)\delta+\delta_v]}(x_v+Z_v(s))>\sum_{u\in N_G^-(v)}c_{uv}.
\end{aligned}\end{equation*}
By Lemma \ref{lem:D_Conv_PointwiseUniform1}, we can take $N_1\ge N_0$ such that for all $n\ge N_1$:
\begin{equation*}\begin{aligned}
\sup_{s\in[m\delta,(m+1)\delta]}(x_v^n+Z_v^n(s))<0\quad\text{or}\quad\inf_{s\in[m\delta,(m+1)\delta]}(x_v^n+Z_v^n(s))>\sum_{u\in N_G^-(v)}c_{uv}^n
\end{aligned}\end{equation*}
for all $v\in B_G(V_0,1)$ such that $v\notin F_{[m\delta,(m+1)\delta]}$, and all $m\in\{0,1,...,M\}$. In other words, for any $t\in[0,(M+1)\delta]$, any $v\in B_G(V_0,1)$ and any $n\ge N_1$, $v$ being fragile subject to $(G_n,c^n,x^n,Z^n)$ at time $t\in[m\delta,(m+1)\delta]$ implies that $v\in F_{[m\delta,(m+1)\delta]}$.

For $n\ge N_1$, we are going to proceed similarly as in Step 2 and use a top-down induction on the tree $\cT(G_n,v_0)$ to show that, for any $v\in\cT(G_n,v_0)$, if $m\in\{0,1,...,M\}$ is such that $\tau_v^n\in[m\delta,(m+1)\delta)$, then there exists $l\in\cL_m$ such that $v\in F_{[m\delta,(m+1)\delta]}^l$. For the base case, we already know that $v_0\in F_{[M\delta,(M+1)\delta]}^{l_M}$. If $\tau_{v_0}^n\in[m\delta,(m+1)\delta)$ for some $m<M$, it follows from Step 3 that $v_0\in F_{[m\delta,(m+1)\delta]}$, which implies the existence of $l$ such that $v_0\in F_{[m\delta,(m+1)\delta]}^l\cap B_G(F_{[M\delta,(M+1)\delta]}^{l_M},1)\neq\emptyset$. For the inductive step, we assume that $v,v'$ are vertices such that $\tau_v^n\in[m\delta,(m+1)\delta)$, $v\in F_{[m\delta,(m+1)\delta]}^{l_v}$, where $l_v\in\cL_m$, $\tau_{v'}^n\in[m'\delta,(m'+1)\delta)$, and that $v'$ is a child of $v$ in the tree $\cT(G_n,v_0)$. If $m'=m$, then $v'\in F_{[m\delta,(m+1)\delta]}\cap B_G(F_{[m\delta,(m+1)\delta]}^{l_v},1)$, and it must hold that $v'\in F_{[m\delta,(m+1)\delta]}^{l_v}$ by the maximal connectedness of the latter. If $m'<m$, then there exists $l$ such that $v'\in F_{[m'\delta,(m'+1)\delta]}^l\cap B_G(F_{[m\delta,(m+1)\delta]}^{l_v},1)$, which implies $l\in\cL_{m'}$ by the definition of $\cL_{m'}$. The induction argument is now completed, which implies that $\cT(G_n,v_0)\subset V_0$.
\end{proof}

\subsection{Convergence of physical solutions}
\label{sec:proof of main2}
Within this subsection, we will use $(\underline X,\underline D)$ to denote the unique physical solution, as non-physical solutions can potentially occur in the argument.

Below, we present and prove a result that is more general than Theorem \ref{thm-main-2}. Here, we treat the driving noises $(Z_v)_{v\in V}$, which are generic noise processes, as part of the inputs to the equations \eqref{eq:1.1}. It turns out that the only technical assumptions that are essential are the $\delta$-robustness of the limiting configuration $(G,c,x,Z)$.

% \begin{theorem}\label{thm:ConvMinSol}
% Assume that $(G_n,c^n,x^n,Z^n)$ drives robust systems for each $n\ge1$, $(G,c,x,Z)$ drives $\delta$-robust systems for some $\delta>0$, and $Z_v$ satisfies the downward crossing property for any $v\in G$. If $\cL(G_n,c^n,x^n,Z^n)\to\cL(G,c,x,Z)$ in $\cP(\cG_*[\RR\times\cD])$ as $n\to\infty$, then $\cL(G_n,\underline X^n,\underline D^n)\to\cL(G,\underline X,\underline D)$ in $\cP(\cG_*[\cD^2])$.
% \end{theorem}

\begin{proof}[Proof of Theorem \ref{thm-main-2}]
\noindent\textbf{Step 1.} We start by showing that the sequence $\{\cL(G_n,c^n,x^n,Z^n,\underline X^n,\underline D^n)\}_{n\ge1}$ is tight in $\cP(\cG_*[\RR\times\cD^3])$. By the Skorokhod Representation Theorem, we can assume without loss of generality that $(G_n,c^n,x^n,Z^n)$ converges almost surely to $(G,c,x,Z)$ in $\cG_*[\RR\times\cD]$ as $n\to\infty$.

Fix any $\varepsilon\in(0,1)$. Since local convergence holds almost surely, we can find a compact subset $\cK_0^\varepsilon$ of $\cG_*[\RR\times\cD]$ together with a sequence $(M_m)_{m\ge1}$ of positive real numbers such that
\begin{equation*}\begin{aligned}
\inf_n\PP\Big[(G_n,c^n,x^n,Z^n)\in\cK_0^\varepsilon,\,\max_{v\in B_G(o,m)}\sum_{u\in N_G^-(v)}c_{uv}\leq M_m\Big]\ge1-\frac{\varepsilon}{2}.    
\end{aligned}\end{equation*}
We define the cumulative loss process
\begin{equation*}\begin{aligned}
\underline L_v^{n}(t):=\sum_{u\in N_{G^n}^-(v)}c_{uv}^n\bone_{\{u\in\underline D_t^n\}},\quad v\in V_n,
\end{aligned}\end{equation*}
which are monotone, right-continuous functions in $D([-1,\infty))$ satisfying $\underline{L}_v^n(t) = 0$ for $t < 0$. For each $m\ge1$, we can take $N_m\in\NN$ such that 
\begin{equation*}\begin{aligned}
\PP\Big[B_{G_n}(o,m+1)=B_G(o,m+1),\quad\max_{v\in B_{G_n}(o_n,m)}\sum_{u\in N_{G_n}^-(v)}c_{uv}^n\leq M_m+1\Big]\ge1-\frac{\varepsilon}{2^{m+2}}   
\end{aligned}\end{equation*}
for all $n\ge N_m$. As each $B_{G_n}(o_n,m+1)$ is a finite graph, by enlarging $M_m$, we can obtain
\begin{equation*}\begin{aligned}
\PP\Big[\max_{v\in B_{G_n}(o_n,m)}\|\underline L_v^n\|_\infty\leq M_m\Big]\ge1-\frac{\varepsilon}{2^{m+1}} 
\end{aligned}\end{equation*}
for all $n\ge1$. Taking a union bound over $m\ge1$, we obtain
\begin{equation*}\begin{aligned}
\inf_n\PP\Big[\max_{v\in B_{G_n}(o_n,m)}\|\underline L_v^n\|_\infty\leq M_m,\quad\forall m\ge1\Big]\ge1-\frac{\varepsilon}{2}   
\end{aligned}\end{equation*}
We further notice that $t\mapsto\bone_{\{v\in\underline D^n_t\}}$ is non-decreasing and bounded by $1$. Combining the above, we obtain 
\begin{equation*}\begin{aligned}
&\inf_n\PP\Big[(G_n,c^n,x^n,Z^n)\in\cK_0^\varepsilon,\,\underline L_v^n\in\cM(M_m)\text{ and }\underline D_v^n\in\cM(1),\,\forall v\in B_{G_n}(o_n,m)\,\forall m\ge1\Big]\\
&\ge1-\varepsilon, 
\end{aligned}\end{equation*}
which implies the tightness of $\{\cL(G_n,c^n,x^n,Z^n,\underline L^n,\underline D^n)\}_{n\ge1}$ in $\cP(\cG_*[\RR\times\cD^3])$ by Theorem \ref{thm:D_Compact_2} and Lemma \ref{lem:MarkedGraph_Comp}. As 
\begin{equation*}\begin{aligned}
\underline X_v^n=x_v^n+Z_v^n-\underline L_v^n,\quad v\in V_n,    
\end{aligned}\end{equation*}
we obtain that $\{\cL(G_n,c^n,x^n,Z^n,\underline X^n,\underline D^n)\}_{n\ge1}$ is tight in $\cP(\cG_*[\RR\times\cD^3])$ by combining Lemma \ref{lem:MarkedGraph_ContMap} Lemma \ref{lem:ContMapThm} and \cite[Lemma 12.7.3]{whitt2002stochastic}.\\

\noindent\textbf{Step 2.} By the Skorokhod Representation Theorem, we can find a $\cG_*[\RR\times\cD^3]$-valued random element $(G,c,x,Z,X,D)$ such that $(G_n,c^n,x^n,Z^n,\underline X^n,\underline D^n)\stackrel{a.s.}{\to}(G,c,x,Z,X,D)$ in $\cG_*[\RR\times\cD^3]$ as $n\to\infty$. In this step, we verify that $(X,D)$ is a solution to equations \eqref{eq:1.1}. As
\begin{equation*}\begin{aligned}
\underline X_v^n(t)&=x_v^n+Z_v^n(t)-\sum_{u\in N_{G_n}^-(v)}c_{uv}^n\bone_{\{u\in D_t^n\}},\quad v\in V_n,
\end{aligned}\end{equation*}
by Theorem \ref{thm:D_Conv_Pointwise}, for all $t\in[-1,\infty)$ such that $t$ is a continuity point of $(X_v(\cdot),\bone_{\{v\in D_\cdot\}})$ for all $v\in V$ (the set of such $t$ is co-countable and thus dense in $[-1,\infty)$), it holds that
\begin{equation*}\begin{aligned}
X_v(t)&=x_v+Z_v(t)-\sum_{u\in N_G^-(v)}c_{uv}\bone_{\{u\in D_t\}},\quad v\in V.
\end{aligned}\end{equation*}
By the right continuity of $t\mapsto X_v(t)$ and $t\mapsto D_t$, the above equation actually holds for all $t\in[-1,\infty)$. To verify that
\begin{equation*}\begin{aligned}
D_t&=\{v\in V:\,\inf_{s\in[0,t]}X_v(s)\leq0\},\quad\forall t\in[-1,\infty),    
\end{aligned}\end{equation*}
it is sufficient to show that $\tau_v:=\inf\{t\ge0:X_v(t)\leq0\}$ satisfies $\tau_v=\lim_{n\to\infty}\underline\tau_v^n$ for any $v\in V$, as
\begin{equation*}\begin{aligned}
\{v\in V:\,\inf_{s\in[0,t]}X_v(s)\leq0\}=\{v\in V:\tau_v\leq t\}   
\end{aligned}\end{equation*}
and
\begin{equation*}\begin{aligned}
D_t=\{v\in V:\lim_{n\to\infty}\underline\tau_v^n\leq t\}    
\end{aligned}\end{equation*}
by Corollary \ref{cor:Conv_JumpTimes}.
To prove $\tau_v=\lim_{n\to\infty}\underline\tau_v^n$, we can use the same argument as in the proof of \cite[Proposition 5.8]{DELARUE2015Spike}, as the limit process $X_v$ satisfies the downward crossing property.\\

\noindent\textbf{Step 3.} We now show that $(X, D)$ coincides with the minimal solution $(\underline{X}, \underline{D})$. By the right-continuity of both $t\mapsto D_t$ and $t\mapsto\underline D_t$, it suffices to show that $\underline\tau_v=\tau_v$ for any $v\in V$, where
\begin{equation*}\begin{aligned}
\underline\tau_v:=\inf\{t\ge0:\underline X_v(t)\leq0\}.  
\end{aligned}\end{equation*}
Note that as $\underline D$ is the minimal solution, we already have $\underline D_t\subset D_t$ for any $t\ge0$ and thus $\underline\tau_v\ge\tau_v$ for any $v\in V$.

We start with the assumption that there exists $v_0\in V$ such that $\underline\tau_{v_0}>\tau_{v_0}$ and aim to arrive at a contradiction. First, let $V_0$ be the finite subset and $N_0\in\NN$ be the threshold given by Lemma \ref{lma:finitetree}. We can find a sufficiently large $k\in\NN$ such that $V_0\subset B_G(o,k)$. Then, there exists $N_1\ge N_0$ such that $B_{G_n}(o_n,k+2)=B_G(o,k+2)$ for all $n\ge N_1$. We take a $\Delta>0$ sufficiently small so that $2\Delta<\underline\tau_v-\tau_v$ for any $v\in B_G(V_0,1)$ such that $\underline\tau_v>\tau_v$. We then take $\varepsilon>0$ sufficiently small so that 
\begin{equation*}\begin{aligned}
\varepsilon<\min_{v\in B_G(V_0,1):\underline\tau_v<\infty}\inf_{s\in[0,\underline\tau_v-\Delta]}\underline X_v(s)\wedge\min_{u,v\in B_G(V_0,1):(u,v)\in E,c_{uv}>0}c_{uv}.
\end{aligned}\end{equation*}
Next, there exists $N_2\ge N_1$ such that for all $n\ge N_2$ and all $v\in B_G(V_0,1)$ such that $\underline\tau_v<\infty$,
\begin{equation*}\begin{aligned}
|x_v^n-x_v|+\sum_{u\in N_v^-(G)}|c_{uv}^n-c_{uv}|<\frac\varepsilon4.
\end{aligned}\end{equation*}
Since each $Z_v(\cdot)$ is right continuous with no upward jumps, there exists $\Delta'<\frac14\Delta$ such that
\begin{equation*}\begin{aligned}
\sup_{s,t\in[0,\tau_v+\Delta],\,t\leq s\leq t+2\Delta'}(Z_v(s)-Z_v(t))_+\leq\frac\varepsilon4 
\end{aligned}\end{equation*}
for any $v\in B_G(V_0,1)$ such that $\tau_v<\infty$. Furthermore, as $\lim_{n\to\infty}\underline\tau_v^n=\tau_v$, by Lemma \ref{lem:D_Conv_Pointwise2} and the assumption that $Z_v(\cdot)$ has no upward jumps, there exists $N_3\ge N_2$ such that, for all $n\ge N_3$ and all $v\in B_G(V_0,1)$, $|\tau_v-\underline\tau_v^n|<\Delta'$ and
\begin{equation*}
(Z_v(\underline\tau_v^n)-Z_v^n(\underline\tau_v^n))_+\leq\frac{\varepsilon}{4}.    
\end{equation*}

The claim is that, for all $v\in V_0$ such that $\underline\tau_v>\tau_v$, there must exist $v'\in N_G^-(v)$ such that $v'$ is a child of $v$ in $\cT(G_n,v)$ for all $n\ge N_3$ and that $\underline\tau_{v'}>\tau_{v'}$. Indeed, consider the default time $(\underline\tau_v^n,k_v^n)$ of $v$ in the system $(\underline X^n,\underline D^n)$. Then
\begin{equation*}\begin{aligned}
0\ge x_v^n+Z_v^n(\underline\tau_v^n)-\sum_{u\in N_G^-(v)}c_{uv}^n\bone_{\{u\in\underline D_{\underline\tau_v^n}^{n(k_v^n-1)}\}}.
\end{aligned}\end{equation*}
However,
\begin{equation*}\begin{aligned}
\varepsilon&<\inf_{s\in[0,\underline\tau_v-\Delta]}\underline X_v(s)\leq\underline X_v(\tau_v+2\Delta')=x_v+Z_v(\tau_v+2\Delta')-\sum_{u\in N_G^-(v)}c_{uv}\bone_{\{u\in\underline D_{\tau_v+2\Delta'}\}}\\
&\leq x_v^n+Z_v^n(\underline\tau_v^n)-\sum_{u\in N_G^-(v)}c_{uv}^n\bone_{\{u\in\underline D_{\underline\tau_v^n}^{n(k_v^n-1)}\}}+|x_v-x_v^n|+(Z_v(\tau_v+2\Delta')-Z_v(\underline\tau_v^n))_+\\
&+(Z_v(\underline\tau_v^n)-Z_v^n(\underline\tau_v^n))_++\sum_{u\in N_G^-(v)}|c_{uv}^n-c_{uv}|+\sum_{u\in N_G^-(v)}c_{uv}\bone_{\{u\in\underline D_{\underline\tau_v^n}^{n(k_v^n-1)}\setminus\underline D_{\tau_v+2\Delta'}\}}\\
&\leq0+\frac34\varepsilon+\sum_{u\in N_G^-(v)}c_{uv}\bone_{\{u\in\underline D_{\underline\tau_v^n}^{n(k_v^n-1)}\setminus\underline D_{\tau_v+2\Delta'}\}},
\end{aligned}\end{equation*}
which implies the existence of $v'\in N_G^-(v)$ such that $v'\in\underline D_{\underline\tau_v^n}^{n(k_v^n-1)}\setminus\underline D_{\tau_v+2\Delta'}$. In other words, $v'$ is a child of $v$ in $\cT(G_n,v)$ and that $\underline\tau_{v'}>\tau_v+2\Delta'$. As $\tau_{v'}-\Delta'<\underline\tau_{v'}^n\leq\underline\tau_v^n<\tau_v+\Delta'$, we obtain the desired property that $\underline\tau_{v'}>\tau_{v'}$.

As a result, we can iteratively extract a sequence $(v_m)_{m\ge0}$ with the property that $\underline\tau_{v_m}>\tau_{v_m}$ and that each $v_{m+1}$ is a child of $v_m$ in $\cT(G_n,v_m)\subset\cT(G_n,v_0)$. This is a contradiction as $\cT(G_n,v_0)\subset V_0$, the latter of which being a finite subset of $V$. Therefore, we obtain $D_t=\underline D_t$, for any $t\ge0$ and thus for $v\in V$,
\begin{equation*}\begin{aligned}
X_v(t)&=x_v+Z_v(t)-\sum_{u\in N_G^-(v)}c_{uv}\bone_{\{u\in D_t\}}=x_v+Z_v(t)-\sum_{u\in N_G^-(v)}c_{uv}\bone_{\{u\in\underline D_t\}}=\underline X_v(t).
\end{aligned}\end{equation*}
\textbf{Step 4.} It follows from Step 1-3 that the sequence $(\cL(G_n,\underline X^n,\underline D^n))_{n\ge1}$ is tight and any of its limit points identifies with $\cL(G,\underline X,\underline D)$, which is unique by Corollary \ref{cor:uniqueness}. Therefore, the whole sequence $\cL(G_n,\underline X^n,\underline D^n)$ converges to $\cL(G,\underline X,\underline D)$ in $\cP(\cG_*[\cD^2])$ as $n\to\infty$.
\end{proof}
\begin{remark}
Following Remark \ref{rem:pathwise1}, the pathwise solution map $\varphi$ induces a map $\Phi$ acting on probability distributions such that $\cL(G,X,D)=\Phi(\cL(G,c,x,Z))$, which is proved to be continuous. As a result, we have actually obtained the continuity of $\varphi$ on the set of configurations which drive $\delta$-robust systems for some $\delta>0$.   
\end{remark}

\subsection{Convergence of empirical measures}\label{sec:proof of main3}
Recall that for a finite graph $G$, its empirical distribution of the paths and default times associated with the physical solution is defined as
\begin{equation*}\begin{aligned}
\mu:=\frac{1}{|V|}\sum_{v\in V}\delta_{(X_v,\tau_v)}. 
\end{aligned}\end{equation*}

\begin{proof}[Proof of Theorem \ref{thm-main-3}]
By assumption, we have
\begin{equation*}\begin{aligned}
\frac{1}{|G_n|} \sum_{v \in G_n} \delta_{\mathcal{C}_v(G_n,c^n,x^n,Z^n)} \to \mathcal{L}(G,c,x,Z) \quad \text{in probability},    
\end{aligned}\end{equation*}
where $\mathcal{C}_v(G_n,c^n,x^n,Z^n)$ denotes the connected component of the marked graph $(G_n,c^n,x^n,Z^n)$ rooted at $v$. Let $\Phi$ be the map that assigns the law of its physical solution to the law of input configuration, which is well-defined and continuous by Theorem~\ref{thm-main-2} at distribution measures of configurations that drive $\delta$-robust systems and have a driving noise component satisfying the downward crossing property. Then,
\begin{equation*}\begin{aligned}
\lim_{n\to\infty}\frac{1}{|G_n|}\sum_{v\in G_n}\delta_{C_v(G_n, X^n,D^n)}&=\lim_{n\to\infty}\Phi\Bigg(\frac{1}{|G_n|}\sum_{v\in G_n}\delta_{C_v(G_n,c^n,x^n,Z^n)}\Bigg)\\
&=\Phi\left(\cL(G,c,x,Z))\right)\\
&=\cL(G,X,D)\quad\text{in probability}.   
\end{aligned}\end{equation*}
The second claim follows from the above combined with Remark \ref{rem:law_root} and Corollary \ref{cor:Conv_JumpTimes}.
\end{proof}

\subsubsection{Systems on regular trees}
As the general theory has been proved, we move on to study some simple examples to illustrate how the general result can be applied. We start with the example $G=(V,E)$ where $V=\ZZ$, $E=\{(i,i+1):\,i\in\ZZ\}$ and $(c_{i,i+1})_{i\in\ZZ}$ are non-negative random variables. The graph $G$ models an infinitely long chain of banks with uni-directional exposure. We assume $(c_{i,i+1},x_i,Z_i)_{i\in\ZZ}$ are i.i.d. across $i\in\ZZ$ and that $(G,c,x,Z)$ drives $\delta$-robust systems for some $\delta>0$ (the latter condition is not restrictive, as the percolation threshold for $G$ is $1$). It follows from Theorem \ref{thm-main-3} that
\begin{equation*}\begin{aligned}
\lim_{n\to\infty}\frac{1}{2n+1}\sum_{i=-n}^n\delta_{\tau_i^n}=\cL(\tau_0)\quad\text{in probability},    
\end{aligned}\end{equation*}
where $\tau_i^n$ is the default time of the $i$-th bank in the system described by graph $G|_{[-n,n]}$, and $\tau_0$ is the default time of the root bank in the system described by graph $G$. We aim to provide a tractable characterization of the distribution of $\tau_0$. Note that
\begin{equation*}\begin{aligned}
X_0(t)=x_0+Z_0(t)-c_{-1,0}\bone_{\{\tau_{-1}\leq t\}},    
\end{aligned}\end{equation*}
$\tau_{-1}$ is independent of $(c_{-1,0},x_0,Z_0)$, and that $\cL(\tau_{-1})=\cL(\tau_0)$, we see that the cumulative distribution function $F_{\tau_0}$ of $\tau_0$ is a fixed point of the map $\Psi:\cP([0,\infty])\to\cP([0,\infty])$ defined as (here we identify a probability measure on $[0,\infty]$ with its CDF)
\begin{equation}\label{eq:tau_dist_iteration1}
\begin{aligned}
\Psi[F]_t&:=\PP\Big[\inf_{s\in[0,t]}\left(x_0+Z_0(s)-c_{-1,0}\bone_{\{\tau_{-1}\leq s\}}\right)\leq0\Big]\quad\tau_{-1}\sim F,\quad\tau_{-1}\indep(c_{-1,0},x_0,Z_0)\\
&=\int_{[0,\infty]}\PP\Big[\inf_{s\in[0,t]}\left(x_0+Z_0(s)-c_{-1,0}\bone_{[r,\infty)}(s)\right)\leq0\Big]\,\mathrm{d}F(r),\quad t\in[0,\infty].
\end{aligned}
\end{equation}
\begin{proposition}\label{prop:DefaultTimeDist0}
$F_{\tau_0}$ equals the minimal fixed point of $\Psi$ restricted to the set of cumulative distribution functions of probability measures on $[0,\infty]$.   
\end{proposition}
The proof will be given in Appendix \ref{sec:app-proofs}.

The same argument can also be applied to the analysis of financial networks modeled by $\mathbb{T}_{k+1}$, the infinite $k+1$-regular tree, in which each vertex (bank) has $1$ out-neighbor (creditor) and $k$ in-neighbors (debtors) and $(c_v,x_v,Z_v)_{v\in\mathbb{T}_{k+1}}$ are i.i.d. across $v\in\mathbb{T}_{k+1}$. We give the following characterization of the distribution of the default time $\tau_0$ where $0$ denotes the root of $\mathbb{T}_{k+1}$. The proof, which we omit for brevity, can be done by truncating the tree at depths $n\to\infty$ and conducting a bottom-up induction on the truncated trees.
\begin{proposition}\label{prop:DefaultTimeDist1}
$F_{\tau_0}$ equals the minimal fixed point of $\Psi_{k}$ restricted to the set of cumulative distribution functions of probability measures on $[0,\infty]$, where     
\begin{equation*}\begin{aligned}
\Psi_{k}[F]_t&:=\PP\Big[\inf_{s\in[0,t]}\Big(x_0+Z_0(s)-\sum_{i=1}^{k}c_{i,0}\bone_{\{\tau_i\leq s\}}\Big)\leq0\Big]\\&\quad\quad(\tau_i)_{1\leq i\leq k}\stackrel{\mathrm{i.i.d.}}{\sim} F,\quad\quad(\tau_i)_{1\leq i\leq k}\indep((c_{i,0})_{1\leq i\leq k},x_0,Z_0)\notag\\
&=\int_{[0,\infty]^{k}}\PP\Big[\inf_{s\in[0,t]}\Big(x_0+Z_0(s)-\sum_{i=1}^{k}c_{i,0}\bone_{[r_i,\infty)}(s)\Big)\leq0\Big]\,\mathrm{d}F_1(r_1)\cdots \mathrm{d}F_{k}(r_{k}).
\end{aligned}\end{equation*}
\end{proposition}
A key observation for such models on trees (connected and acyclic graphs) is that, for a fixed vertex, there is mutual independence between the driving noise of that vertex and the default times of its in-neighbors. In general, such independence does not hold anymore if the underlying graph is cyclic. In the latter case, the analysis of the default time distribution may require adapting the theory of \emph{local equations} developed in \cite{lacker2023localeq} and will be left for future work.

\subsubsection{A preliminary numerical experiment}
We remark that Propositions \ref{prop:DefaultTimeDist0} and \ref{prop:DefaultTimeDist1} provide a straightforward method to numerically approximate the distribution of the default time of a representative bank via iterating the map $\Psi_k$ and evaluating the probabilities with Monte-Carlo simulations. We plot the cumulative distribution functions of the default times corresponding to $k\in\{1,2,5,10\}$ as well as several sample paths of the process $X(t)$. In the experiment, we take $x_0$ to be distributed as the absolute value of $\mathcal{N}(10,4^2)$, $Z_0$ the standard Brownian motion, and $c=\frac{10}{k}$ for a fair comparison (so that the total exposure is kept constant). The number of iteration is taken to be $1000$, the total time horizon is $50$, the number of time steps is $1000$, and the number of Monte-Carlo samples is $5000$. The Python code for generating the above plot is attached in the supplementary materials.

%  \begin{figure}
%      \FIGURE
% {\includegraphics[width=0.45\textwidth]{plot11.png}\includegraphics[width=0.45\textwidth]{plot12.png}\label{fig:1}}
% {\label{fig1}}
% {}
% \end{figure}
\begin{figure}[t]
    \centering
    \includegraphics[width=0.45\linewidth]{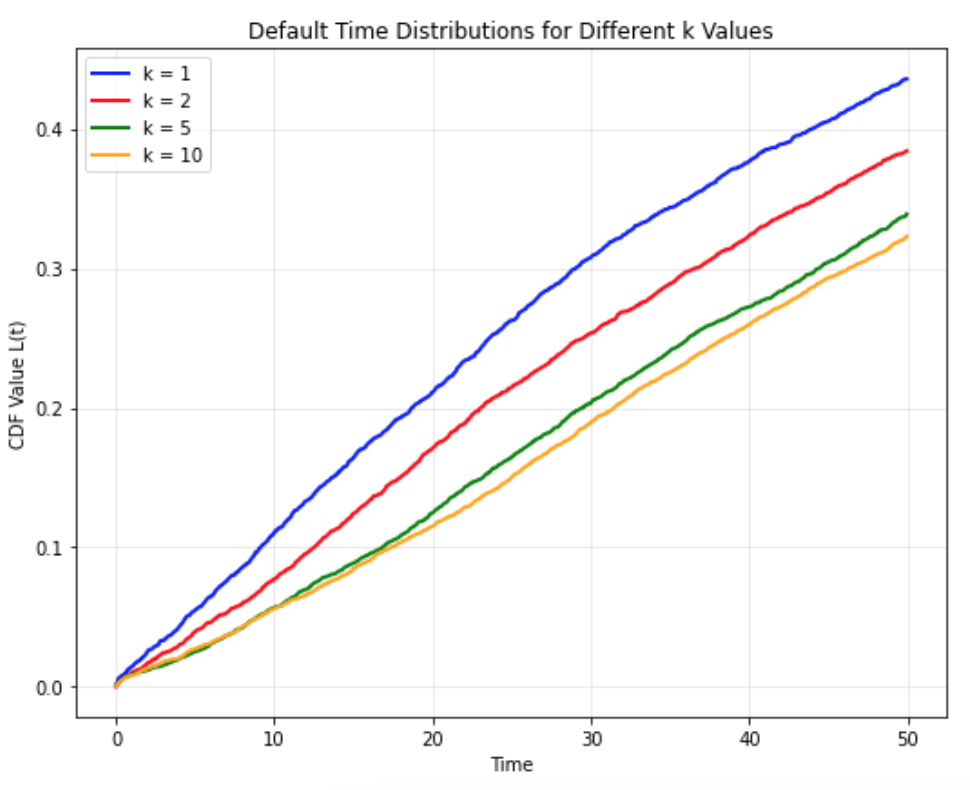}%
    \includegraphics[width=0.45\linewidth]{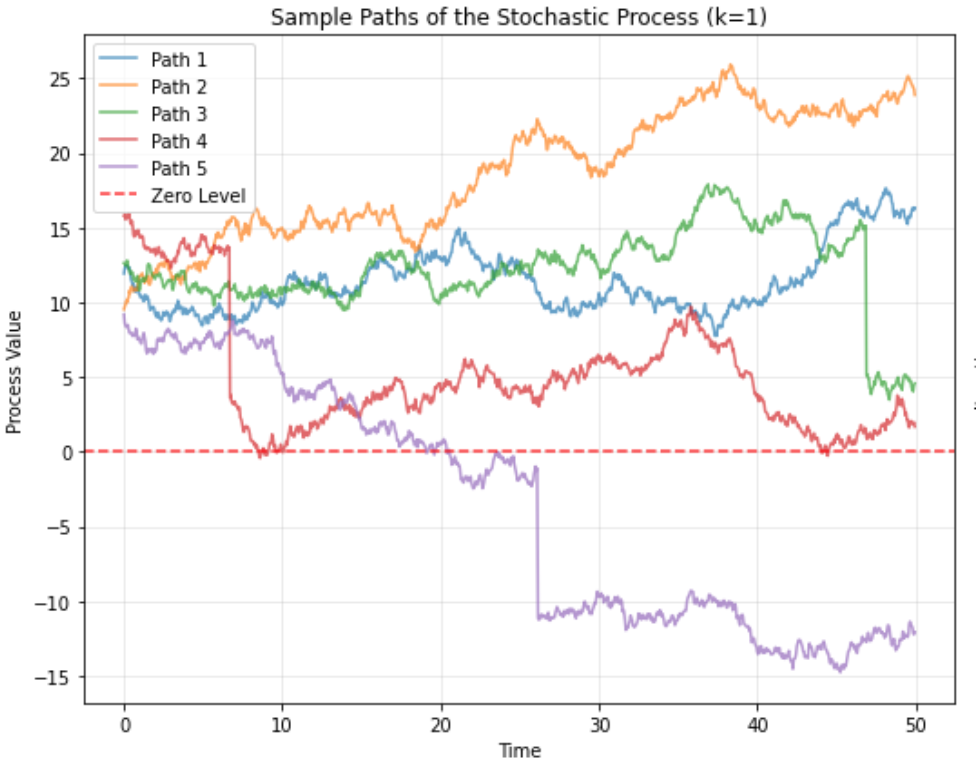}
    
    \label{fig:1}
\end{figure}

As demonstrated by the numerical results, the specific problem configuration lies in the regime in which increased connectedness lowers the fraction of defaults. A comprehensive theoretical analysis of the complex interplay between the overall health and the topological structure of financial networks remains challenging and represents an important direction for future research.

% \subsubsection{Systems on general graphs}

% \subsubsection{Comparison with the mean-field limit}
% If we take $G_n$ to be the complete graph on $n$ vertices and let $c_{ij}^n=\frac{1}{n-1}$ for any $i\neq j$, then the empirical cumulative distribution of default times
% \begin{equation*}\begin{aligned}
% \widehat\Lambda_t^n:=\frac{1}{n}\sum_{i=1}^n\bone_{\{\tau_i^n\leq t\}}
% \end{aligned}\end{equation*}
% corresponding to the physical solution to
% $\eqref{eq:1.1}$ and \eqref{eq:1.2} converges in probability to $\Lambda$, which is the minimal solution to the mean-field equation
% \begin{align}
% X(t)&=x_0+B(t)-\Lambda_t\label{eq:mean-field1}\\
% \Lambda_t&=\PP[\tau\leq0],\quad\tau:=\inf\{t\ge0:X_t\leq0\}\label{eq:mean-field2},
% \end{align}
% provided that there exists a unique physical solution to the later. In a similar manner to the systems on graphs, the mean-field equations \eqref{eq:mean-field1} and \eqref{eq:mean-field2} are equivalent to the fixed-point equation $\cL(\tau)=\Gamma_{\mathrm{MF}}[\cL(\tau)]$, where 
% \begin{align}\label{eq:tau_dist_iteration_MF}
% \Gamma_{\mathrm{MF}}[\nu]([0,t]):=\PP\Big[\inf_{s\in[0,t]}\left(x_0+B(s)-\nu([0,s])\right)\leq0\Big],\quad t\in[0,\infty].    
% \end{align}
% It is insightful to compare the iteration map $\Psi$ defined in equation \eqref{eq:tau_dist_iteration1} with $\Gamma_{\mathrm{MF}}$.

% \color{blue}{By Proposition3.20, $\cL(\tau_0)$ is the maximal fixed point of $\Psi$, constructed by $\Bar{\nu}$=$\lim_{n\rightarrow \infty}\Psi^n[\delta_\infty]$, with $\Psi$ decreasing. While the mean field map $\Gamma_{MF}$ is increasing. }
\subsection{Connections to delayed-loss models}
\label{sec:delayedmodel}
As the model studied in this paper features instantaneous propagation of losses, a natural question is how
the physical solution compares with the solution to an alternative class of models which introduce a time lag in the propagation of losses between connected entities. Many of the latter, in their general form, can be described by:
\begin{equation}\label{eq:delay1}
\begin{aligned}
X_v^\lambda(t)&=x_v+Z_v(t)-\sum_{u\in N_G^-(v)}c_{uv}\lambda_{uv}(t,X_u^\lambda),\quad v\in V\\
\tau_v&:=\tau(X_v^\lambda):=\inf\{t\ge0:\,X_v^\lambda\leq0\},
\end{aligned}
\end{equation}
where $(\lambda_{uv})_{(u,v)\in E}$ is a family of (possibly random) functions that satisfy
\begin{equation*}\begin{aligned}
&t\mapsto\lambda_{uv}(t,x)\in\cD\text{ for any fixed $x\in\cD$,}\\
&t\mapsto\lambda_{uv}(t,x)\text{ is non-decreasing,}\\
&\lambda_{uv}(t)=0\quad\forall t\in(-\infty,\tau(x)),\quad\lambda_{uv}(t)\leq1\quad\forall t\in[\tau(x),\infty)\\
&\quad\quad\quad\text{ where }\tau(x):=\inf\{t\ge0:x(t)\leq0\},\\
&\text{system \eqref{eq:delay1} admits a unique solution.}
\end{aligned}\end{equation*}
Examples of such models include
\begin{enumerate}
    \item Interaction through elastic stopping times in \cite{Hambly2022Elastic}:
    \begin{equation*}\begin{aligned}
    \lambda_{uv}(t,x):=\bone_{\{\inf_{s\in[0,t]}x(s)\leq-\xi_u\}},    
    \end{aligned}\end{equation*}
    where $(\xi_v)_{v\in V}$ are i.i.d. exponential random variables with rate parameter $\kappa>0$. $\lambda_{uv}(\cdot,x)$ converges in distribution to $\bone_{[\tau(x),\infty)}(\cdot)$ as $\kappa\to\infty$.
    \item Regularized impact models in \cite{hambly_spde_2019,BURZONI2023206,Inglis2015Neuron}:
    \begin{equation*}\begin{aligned}
    \lambda_{uv}(t,x)=\int_0^{(t-\tau)_+}k_\varepsilon(s)\,\mathrm{d}s,    
    \end{aligned}\end{equation*}
    where $k_\varepsilon(\cdot)=\frac{1}{\varepsilon}k(\frac{\cdot}{\varepsilon})$ and $k$ is a non-negative function compactly supported in $[0,\infty)$ with the property that $\int_0^\infty k(s)\,\mathrm{d}s=1$. $\lambda_{uv}(\cdot,x)$ converges in distribution to $\bone_{[\tau(x),\infty)}(\cdot)$ as $\varepsilon\downarrow0$.
    \item Default intensity models in \cite{Spiliopoulos02112019}:
    \begin{equation*}\begin{aligned}
    \lambda_{uv}(t,x)=\bone_{\{\int_0^{(t-\tau)_+}r_{uv}(s)\ge \xi_{uv}\}}(t),    
    \end{aligned}\end{equation*}
    where $r_{uv}$ is some stochastic intensity process and $(\xi_{uv})_{(u,v)\in E}$ are i.i.d. exponential random variables.
\end{enumerate}
\begin{theorem}\label{thm:ConvDelay}
    Let $(\lambda_{uv}^n)_{(u,v)\in E}$ be a sequence of random functions such that, for every $(u,v)\in E$, the processes $\lambda_{uv}^n(\cdot,x)$ converges locally uniformly, around paths satisfying the downward crossing property, almost surely to $\bone_{[\tau(x),\infty)}(\cdot)$ in $\cD$ as $n\to\infty$, in the following sense:
\begin{equation}\label{eq:lambda_conv1}
\lim_{n\to\infty} d_{M_1}\big(\lambda_{uv}^n(\cdot,x^n),\bone_{[\tau(x^n),\infty)}(\cdot)\big)=0
\qquad \text{a.s.},
\end{equation}
where
\begin{equation*}
x^n(t)=Z(t)-l^n(t),
\end{equation*}
and $Z$ is a fixed process satisfying the downward crossing property in Assumption~\ref{ass:Main}, and $(l^n)_{n\ge1}$ is any fixed sequence of uniformly bounded, nondecreasing stochastic paths. Then $(G,X^{\lambda^n},D^n)$ converges almost surely to $(G,X,D)$ in $\cG_*[\cD^2]$ as $n\to\infty$. 
\end{theorem}
The proof will be given in Appendix \ref{sec:app-proofs}.

\section{Concluding Discussion}
Our analysis of fragility and robustness in default cascade models on infinite graphs provides insights into systemic risk in financial networks. In particular, the results highlight how network topology and local interactions fundamentally shape the resilience of banking systems to shocks. Highly interconnected institutions can serve as amplifiers of contagion, triggering widespread default cascades. This underscores the importance of incorporating network-based assessments into regulatory stress-testing frameworks, moving beyond institution-level risk metrics to system-wide vulnerability analyses.

The model also reveals that even minimal initial shocks can propagate extensively when structural fragility is present. This finding supports the case for targeted preventive measures—such as capital or liquidity buffers—for institutions that play a critical role in the network. More broadly, the framework offers a rigorous foundation for identifying systemic vulnerabilities and guiding preemptive strategies aimed at minimizing the probability and severity of financial crises.

\subsection{Extensions to models with major banks}\label{sec:MajorBanks}
Admittedly, real financial networks often contain a small number of major institutions—such as central banks or globally systemic financial institutions—that interact with a large fraction of market participants. Such entities fall outside the present framework of local interactions on locally finite graphs, since they effectively introduce vertices of infinite degree and thereby distort the underlying graph distance, and hence the geometry, of the network. Nevertheless, the framework can be adapted by separating the global influence of these major institutions from the remaining local interactions among smaller banks. 

One possible way to model such global effects is through the common noise component $Z^0$ as introduced in item 3 of Lemma~\ref{lem:robust1}. However, this approach has the drawback that it does not capture feedback from the system: in particular, it cannot account for how defaults of minor banks affect the financial health of major institutions.

Below we illustrate another more systematic way to incorporate this separation within our model. Minor banks continue to interact among themselves as in the previous formulation, but each minor bank also interacts simultaneously with all major banks. Thus, from the perspective of the minor banks, the major banks act as additional global (weak) counterparties. When the number of minor banks tends to infinity, the effect of the minor-bank population on any fixed major bank becomes diminishing (and potentially mean-field like), and therefore the major banks’ interaction structure could differ from that of the minor banks, depending on the specific modeling assumptions.

Let $[N]:=\{1,\ldots,N\}$ denote the set of the major banks. Enlarge the original locally finite graph $G$ by $\tilde G:=(\tilde V, \tilde E)$ where $\tilde V:=V\cup [N]$ and 
\begin{equation*}
\tilde E:=E\cup\{(i,v),(v,i):v\in V,i\in[N]\}\cup\{(i,j)\in[N]^2:i\neq j\}.
\end{equation*}
The corresponding interacting particle system on $\tilde G$ is
\begin{equation}\label{eq:Major-Minor}
    \begin{aligned}
        X_v(t)&=x_v+Z_v(t)-\sum_{u\in N_G^-(v)}c_{uv}\bone_{\{u\in D_t\}}-\sum_{i=1}^N c_{iv}\bone_{\{i\in D_t\}},\quad v\in V\\
    X_i(t)&=x_i+Z_i(t)-\sum_{v\in V} c_{vi} \bone_{\{v\in D_t\}}-\sum_{j=1}^N c_{ji}\bone_{\{j\in D_t\}},\quad i\in [N] \\
D_t&=\{v\in \tilde V:\,\inf_{s\in[0,t]}X_v(s)\leq0\},  
    \end{aligned}
\end{equation}
where it is necessary to impose summability conditions 
\begin{equation*}
\max_{i\in [N]}\, \sum_{v\in V} \,(c_{vi}+c_{iv})<\infty.
\end{equation*}
With a suitable refinement of the arguments developed in this chapter (in particular, the percolation argument should still be performed on the graph $G$ rather than $\tilde G$), it should be possible to extend the main results to the system \eqref{eq:Major-Minor}. We leave this extension for future work.

\subsection{Other future directions}
We comment on several promising directions for future research that emerge from this work. A first avenue involves enriching the model with detailed balance sheet structures, moving beyond a scalar health variable to explicitly account for asset-liability compositions, claim seniority, collateralization, and regulatory capital buffers. This would allow for more granular estimation of creditor losses and more realistic modeling of contagion mechanisms.

Second, introducing dynamic network features—such as institutional entry and exit, multiple defaults per node, and adaptive exposure reallocation—would better capture the evolving nature of financial systems. Incorporating regulatory constraints, market frictions, and mechanisms for resolution and re-entry would further enhance the framework’s ability to reflect real-world institutional behavior.

Third, embedding strategic interaction into network formation would allow institutions to optimize exposures under risk-return trade-offs, subject to systemic constraints. A game-theoretic extension could capture how such strategic decisions interact and how macro-prudential tools influence network architecture and collective outcomes.

Finally, empirical calibration using data on interbank exposures and historical default events is essential for improving the model’s practical relevance. Estimating unobserved network structures, calibrating parameters to match crisis dynamics, and validating the model against real contagion episodes would enhance its usefulness for central banks and financial regulators engaged in stress testing and systemic risk monitoring.

\appendix
\section{The space $\cD:=D([-1,\infty))$ and the $M_1$-topology}\label{sec:cD}

\subsection{The space $D([a,b])$}
For $a, b \in \mathbb{R}$ with $a < b$, we denote by $D([a,b])$ the space of functions $f : [a,b] \to \mathbb{R}$ that are right-continuous at every $t \in [a,b)$, possess left limits at every $t \in (a,b]$, and satisfy $f(b-) = f(b)$. 

We now introduce a metric that induces the $M_1$ topology on $D([a,b])$, following the presentation in \cite{delarue2015global}. For a function $f\in D([a,b])$, let $\cG_f$ denote the completed graph of $f$:
\begin{equation*}\begin{aligned}
\cG_f:=\{t\in[a,b]:x\in[f(t-),f(t)]\},    
\end{aligned}\end{equation*}
where $[f(t-),f(t)]$ is the non-ordered closed segment between $f(t-)$ and $f(t)$, and we manually set $f(a-):=f(a)$. An order on $\cG_f$ can be defined as follows: for $(t_1, x_1), (t_2, x_2) \in \mathcal{G}_f$, we write $(t_1, x_1) \leq (t_2, x_2)$ if either $t_1 < t_2$, or $t_1 = t_2$ and $|f(t_1-) - x_1| \leq |f(t_1-) - x_2|$.

 A \emph{parametric representation} of $\cG_f$ is a continuous map $(r,u):[a,b]\to\cG_f$ that is surjective and non-decreasing with respect to the above order. Let $\cR_f$ denote the set of all such parametric representations of $\cG_f$. For $f_1,f_2\in D([a,b])$, the $M_1$ distance between them is defined as
\begin{equation*}\begin{aligned}
d_{M_1}(f_1,f_2):=\inf_{(r_i,u_i)\in\cR_{f_i},i=1,2}(\|r_1-r_2\|_\infty\vee\|u_1-u_2\|_\infty).  
\end{aligned}\end{equation*}
The space $D([a,b])$, equipped with the $M_1$ topology, is a Polish space. Moreover, its Borel $\sigma$-field coincides with the $\sigma$-field generated by the evaluation mappings $(f\mapsto f(t))_{t\in[a,b]}$.

\subsection{The space $\cD:=D([-1,\infty))$}
We denote by $D([-1,\infty))$ the space of functions $f:[-1,\infty)\to\RR$ that are right-continuous at all $t\in[-1,\infty)$ and have left limits at all $t\in(-1,\infty)$. For $f\in D([-1,\infty))$ and any $t>0$, we denote by $f|_{[-1,t-]}$ the restriction of $f$ to $[-1,t]$ with the value at $t$ replaced by $f(t-)$, so that $f|_{[-1,t-]}\in D([-1,t])$. The $M_1$ distance on it can be defined by
\begin{equation*}\begin{aligned}
d_{M_1}(f_1,f_2):=\int_0^\infty e^{-t}d_{M_1}(f_1|_{[-1,t-],},f_2|_{[-1,t-]})\,\mathrm{d}t. 
\end{aligned}\end{equation*}
In other words, a sequence $(f_n) \subset D([-1, \infty))$ converges to $f$ in the $M_1$ topology if and only if there exists a sequence $(t_m)_{m\ge0}\uparrow\infty$ (possibly depending on $f$) such that $f_n|_{[-1,t_m-]}\to f|_{[-1,t_m-]}$ in $D([-1,t_m])$ as $n\to\infty$, for each $m\ge0$. Again, $D([-1,\infty))$ endowed with the $M_1$ topology is a Polish space, and its Borel $\sigma$-field coincides with the $\sigma$-field generated by the evaluation mappings $(f\mapsto f(t))_{t\in[-1,\infty)}$.

In this paper, the set-valued process $[-1,\infty)\ni t\mapsto D_t$ such that $D_t\subset V$ for all $t$ will be identified with $(t\mapsto\bone_{\{v\in D_t\}})_{v\in V}\in\cD^V$, which can be thought of as a collection of $\cD$-marks on the vertex set $V$. This identification is justified by the following lemma, which directly follows the definition.
\begin{lemma}
The set-valued process $t\mapsto D_t$ is right continuous, if and only if $t\mapsto\bone_{\{v\in D_t\}}$ is right continuous for any $v\in V$, if and only if $(G,(t\mapsto\bone_{\{v\in D_t\}})_{v\in V})\in\cG_*[\cD]$.    
\end{lemma}

\subsection{Towards compactness in $\cD$}
The next result provides a sufficient condition for compactness of subsets of $D([-1, \infty))$ consisting of monotone and uniformly bounded paths. This compactness criterion will be useful in proving tightness of empirical processes arising from our particle system. 

\begin{theorem}\label{thm:D_Compact_2}
Fix any sequence $(M_m)_{m\ge1}$ of non-decreasing positive real numbers. The following subset of $D([-1,\infty))$ has compact closure in the $M_1$ topology:
\begin{equation*}\begin{aligned}
\cM((M_m)_{m\ge1}):=\{&f\in D([-1,\infty)):\,f=0\,\,\text{on }[-1,0),\\
& \text{$f$ is monotone and }\sup_{[-1,m]}|f|\leq M_m,\quad\forall m\ge1\}.    
\end{aligned}\end{equation*}
In particular, for any $M \in (0, \infty)$, the set
\begin{equation*}\begin{aligned}
\cM(M)&:=\{f\in D([-1,\infty)):\,f=0\,\,\text{on }[-1,0),\,\text{$f$ is monotone and }\sup_{[-1,\infty)}|f|\leq M\}.    
\end{aligned}\end{equation*}    
has compact closure in the $M_1$ topology.
\end{theorem}
\begin{proof}
We take any sequence $f_k\in\cM((M_m)_{m\ge1})$ and we need to show that $(f_k)_{k\ge1}$ has a limit point in $D([-1,\infty))$. For each $m\ge1$, we first define the auxiliary function
\begin{equation*}\begin{aligned}
f_k^m(t):=f_k(t)\bone_{[-1,m]}(t)+f_k(m)\bone_{(m,m+1]}(t),    
\end{aligned}\end{equation*}
which is monotone and remains constant on $[-1,0)$ and on $(m,m+1]$. By Theorem \cite[Theorem 12.12.2]{whitt2002stochastic}, there exists $f_\infty^m\in D([-1,m+1])$ such that $f_k^m$ converges to $f_\infty^m$ in $D([-1,m+1])$ as $k\to\infty$ along some subsequence, which is potentially a further subsequence of the subsequence corresponding to $m-1$. In particular, by \cite[Corollary 12.9.1]{whitt2002stochastic}, we can find a continuity point $t_m\in(m-1,m]$ of $f_\infty^m$ such that $f_k|_{[-1,t_m]}=f_k^m|_{[-1,t_m]}$ converges to $f_\infty^m|_{[-1,t_m]}$ in $D([-1,t_m])$ as $k\to\infty$ along that subsequence (note also that $t_m\uparrow\infty$ by construction). Combining the diagonal argument, Theorem \ref{thm:D_Conv_Pointwise} and the right-continuity of the limit function, we can find an $f_\infty\in D([-1,\infty))$ such that $f_k$ converges to $f_\infty$ in $D([-1,\infty))$ as $k\to\infty$ along some subsequence.
\end{proof}

\section{Auxiliary Technical Results}
\label{sec:app-technical}
The following lemma is a version of Continuous Mapping Theorem that will be useful for us.
\begin{lemma}\label{lem:ContMapThm}
Let $(\mu_n)_{n\ge1}\subset\cP(\cY)$ be a tight sequence of probability measures on the metric space $\cY$, and let $f:\cY\to\cY'$ be a continuous map between metric spaces. Then the sequence of push-forward measures $(f_\#\mu_n)_{n\ge1}\subset \cP(\cY')$ is also tight.
\end{lemma}
\begin{proof}
For any $\varepsilon>0$, there exists a compact subset $K_\varepsilon$ of $Y$ such that $\inf_n\mu(K_\varepsilon)\ge1-\varepsilon$. Since $f$ is continuous, $f(K_\varepsilon)$ is also compact as a subset of $\cY'$. Now, $f_\#\mu(f(K_\varepsilon))=\mu(f^{-1}f(K_\varepsilon))\ge\mu(K_\varepsilon)\ge1-\varepsilon$ for any $n$.
\end{proof}

\begin{proof}[Proof of Lemma \ref{lem:MarkedGraph_Comp}]
Let $((G_n,c^n,y'^n,y^n))_{n\ge1}\subset\cK$ be any sequences. We will show that it admits a convergent subsequence. First, as $\cK_0$ is compact, there exists $(G_\infty,c^\infty,y'^\infty)\in\cG_*$ such that $(G_n,c^n,y'^n)$ converges to $(G_\infty,c^\infty,y'^\infty)$ as $n\to\infty$ along some subsequence, which, without loss of generality, we assume to be the original sequence. For any $m\in\NN$, we can take $N_m\in\NN$ sufficiently large so that $B_{G_n}(o_n,m)=B_G(o,m)$ for all $n\ge N_m$. Now that $y_v^n\in K_m$ for any $v\in B_G(o,m)$ and any $m\in\NN$, by using the diagonal argument, we can obtain $y^\infty\in\cY^V$ such that $y_v^n$ converges to $y_v^\infty\in K_m$ for any $v\in B_G(o,m)$ and any $m\in\NN$ as $n\to\infty$ along some subsequence. In particular, $(G_n,c^n,y'^n,y^n)$ converges locally to $(G_\infty,c^\infty,y'^\infty,y^\infty)$ in $\cG_*[\cY'\times\cY]$ as $n\to\infty$ along the same subsequence.
\end{proof}

The following theorem is an easy extension of \cite[Theorem 12.4.1]{whitt2002stochastic} to $D([-1,\infty))$.

\begin{theorem}\label{thm:D_Conv_Pointwise}
Suppose that the sequence $(f_n)_{n\ge0}$ converges to $f$ in $D([-1,\infty))$ in the $M_1$ topology. Then for all points $t\in[-1,\infty)$ at which $f$ is continuous, it holds that
\begin{equation*}\begin{aligned}
\lim_{\delta\downarrow0}\limsup_{n\to\infty}\sup_{s\in[(-1)\vee(t-\delta),t+\delta]}|f_n(s)-f(s)|=0.    
\end{aligned}\end{equation*}
\end{theorem}

\begin{corollary}\label{cor:Conv_JumpTimes}
Suppose there exists a sequence $(\tau_n)_{n\ge0}\subset[0,\infty]$ such that $(f_n:t\mapsto\bone_{[\tau_n,\infty)}(t))_{n\ge0}$ converges to some $f$ in $D([-1,\infty))$. Then $\tau:=\lim_{n\to\infty}\tau_n\in[0,\infty]$ exists and $f(t)=\bone_{[\tau,\infty)}(t)$. 
\end{corollary}
\begin{proof}
As each $f_n$ takes value in $\{0,1\}$ and is non-decreasing, its limit $f$ must take value in $\{0,1\}$ and be non-decreasing by using the fact that the continuity points are dense and $f$ is right-continuous. For any continuity point $t$ of $f$, we know from Theorem \ref{thm:D_Conv_Pointwise} that $f(t)=\liminf_{n\to\infty}f_n(t)=0$ if $t<\limsup_{n\to\infty}\tau_n$, and $f(t)=\limsup_{n\to\infty}f_n(t)=1$ if $t>\liminf_{n\to\infty} \tau_n$. Combining the two cases, we force that $\liminf_{n\to\infty}\tau_n=\limsup_{n\to\infty}\tau_n$ and thus $\tau:=\lim_{n\to\infty}\tau_n$ exists and also that $f(t)=\bone_{[\tau,\infty)}(t)$ by the right-continuity of $f$.
\end{proof}

\begin{corollary}
Suppose $D^n$ is a sequence of non-decreasing right-continuous set-valued processes such that $(G_n,D^n)$ converges locally to $(G,F)$ in $\cG_*[\cD]$. Then there exists a non-decreasing set-valued process $D$ such that $F_v(t)=\bone_{\{v\in D_t\}}$ for any $v\in V$.    
\end{corollary}

The following two Lemmas characterize the ``uniform convergence" entailed by convergence in the $M_1$ topology.
\begin{lemma}\label{lem:D_Conv_Pointwise2}
Suppose that the sequence $(f_n)_{n\ge0}$ converges to $f$ in $D([-1,\infty))$ in the $M_1$ topology, and $t_n\to t\ge0$. Then
\begin{equation*}
f(t-)\wedge f(t)\leq\liminf_{n\to\infty}f_n(t_n-)\wedge f_n(t_n)\leq \limsup_{n\to\infty}f_n(t_n-)\vee f_n(t_n)\leq f(t-)\vee f(t).    
\end{equation*}
\end{lemma}
\begin{proof}
For any $\delta>0$, we can take $\xi_1\in(t-\delta,t)$ and $\xi_2\in(t,t+\delta)$ that are continuity points of $f$. Noting the definition of $w_s(f,t,\delta)$ as in \cite[Equation (12.4.4)]{whitt2002stochastic}, we get
\begin{equation*}
f_n(\xi_1)\wedge f_n(\xi_2)-w_s(f_n,t,\delta)\leq f_n(t_n-), f_n(t_n)\leq f_n(\xi_1)\vee f_n(\xi_2)+w_s(f_n,t,\delta). 
\end{equation*}
Taking $n\to\infty$ and then $\delta\downarrow0$, we get the desired result by \cite[Theorem 12.5.1 (iv)]{whitt2002stochastic}.
\end{proof}
\begin{lemma}\label{lem:D_Conv_PointwiseUniform1}
Suppose that the sequence $(f_n)_{n\ge0}$ converges to $f$ in $D([-1,\infty))$ in the $M_1$ topology, and that $[a',b']\subset(a,b)\subset[-1,\infty)$. Then
\begin{equation*}
\inf_{t\in[a,b]}f(t)\leq\liminf_{n\to\infty}\inf_{t\in[a',b']}f_n(t)\leq \limsup_{n\to\infty}\sup_{t\in[a',b']}f_n(t)\leq\sup_{t\in[a,b]}f(t).     
\end{equation*}
\end{lemma}
\begin{proof}
Take $\delta>0$ sufficiently small so that $[a'-\delta,b'+\delta]\subset[a,b]$. We can find a partition $t_0<t_1<...<t_M$ consisting of continuity points of $f$ such that $t_0\in[a'-\delta,a']$, $t_M\in[b',b'+\delta]$, and $t_{i+1}-t_i\leq\delta$. For any $t\in[a',b']$, let $i$ be such that $t\in[t_i,t_{i+1}]$. Noting the definition of $w_s(f,\delta)$ as in \cite[Equation (12.5.1)]{whitt2002stochastic} (we can first restrict each $f_n$ to the compact interval $[-1,b+\delta]$, to avoid taking supremum of $w_s(f_n,t,\delta)$ over the non-compact interval $[-1,\infty)$), we get
\begin{equation*}
\min\{f_n(t_i),f_n(t_{i+1})\}-w_s(f_n,\delta)\leq f_n(t)\leq\max\{f_n(t_i),f_n(t_{i+1})\}+w_s(f_n,\delta).    
\end{equation*}
As a result,
\begin{equation*}
\min_{0\leq i\leq M}f_n(t_i)-w_s(f_n,\delta)\leq\inf_{t\in[a',b']}f_n(t)\leq\sup_{t\in[a',b']}f_n(t)\leq\max_{0\leq i\leq M}f_n(t_i)+w_s(f_n,\delta).     
\end{equation*}
Taking $n\to\infty$ and then $\delta\downarrow0$, we get the desired result by combining \cite[Theorem 12.5.1 (iv)]{whitt2002stochastic} and Theorem \ref{thm:D_Conv_Pointwise}.
\end{proof}

\section{Remaining Technical Proofs}\label{sec:app-proofs}
\begin{proof}[Proof of Proposition \ref{prop:DefaultTimeDist0}]
Let $F^0(t):= \bone_{\{t=\infty\}}$ (which is the CDF of the Dirac measure $\delta_\infty$) and define the sequence $F^{n+1} := \Psi[F^n]$. As the map $\Psi$ preserves stochastic dominance in the sense that $\Psi[F]\leq\Psi[\tilde F]$ for any two CDFs with $F\leq\tilde F$ and that $F^0\leq F^1$, we see that $(F^n)_{n\ge0}$ is a non-decreasing sequence and hence the limit
\begin{equation*}\begin{aligned}
\underline F:= \lim_{n \to \infty} F^n
\end{aligned}\end{equation*}
exists in the sense of weak convergence of CDFs and gives the minimal fixed point of $\Psi$. To interpret this iteration probabilistically, define the truncated systems $G_n = G|_{[-n,n]}$ with $V_n := \{ i \in \mathbb{Z} : -n\leq i\leq n \}$ and $E_n := \{ (i, i+1): -n\leq i<n \}$. Let the roots of $G_n$ and $G$ be the vertex $0$, and define $(c^n, x^n, Z^n)$ on $G_n$ by the corresponding restriction of $(c,x,Z)$. 
Apparently, $(G_n,c^n,x^n,Z^n)$ converges almost surely to $(G,c,x,Z)$ in $\cG_*[\RR\times\cD]$ as $n\to\infty$, which, in particular, implies $\cL(\tau_0^n)\to\cL(\tau_0)$ as $n\to\infty$ by Theorem \ref{thm-main-2} and Corollary \ref{cor:Conv_JumpTimes}. As the vertex $-n$ has no in-neighbors in $G_n$, we see that 
\begin{equation*}\begin{aligned}
\PP[\tau_{-n}^n\leq t]=\PP\Big[\inf_{s\in[0,t]}(x_{-n}+Z_{-n}(s))\leq0\Big],
\end{aligned}\end{equation*}
that is, $F_{\tau_{-n}^n}=\Psi[F^0]$. Moreover, as
\begin{equation*}\begin{aligned}
X_{i+1}^n(t)=x_{i+1}+Z_{i+1}(t)-c_{i,i+1}\bone_{\{\tau_{i}^n\leq t\}}
\end{aligned}\end{equation*}
and that $\tau_i^n$ is independent of $(c_{i,i+1},x_{i+1},Z_{i+1})$, the latter of which has the same distribution as $(c_{-1,0},x_0,Z_0)$, it holds that $F_{\tau_{i+1}^n}=\Psi[F_{\tau_i^n}]$ for $i\ge-n$. As a result, we obtain the relation $F_{\tau_0^n}=\Psi^n[F_{\tau_{-n}^n}]=F^{n+1}$, and the claim is proved.
\end{proof}

\begin{proof}[Proof of Theorem \ref{thm:ConvDelay}]
For each $n\ge1$, we point out that the unique solution to \eqref{eq:delay1} can be obtained by an iteration procedure similar to that in the construction of minimal solutions: $(G,X^{\lambda^n},D^n)=\lim_{N\to\infty}(G,X^{\lambda^n,N},D^{n,N})$, where  
\begin{equation*}\begin{aligned}
&X_v^{\lambda^n,0}(t)=x_v+Z_v(t),\quad
X_v^{\lambda^n,N+1}(t)=x_v+Z_v(t)-\sum_{u\in N_G^-(v)}c_{uv}\lambda_{uv}^n(t,X_u^{\lambda^n,N})\\
&D_t^{n,0}=\emptyset,\quad
D_t^{n,N+1}=\{v\in V:\,\inf_{s\in[0,t]}X_v^{\lambda^n,N}(s)\leq0\}.
\end{aligned}\end{equation*}
As
\begin{equation*}\begin{aligned}
X_v^{\lambda^n,N+1}(t)&=x_v+Z_v(t)-\sum_{u\in N_G^-(v)}c_{uv}\lambda_{uv}^n(t,X_u^{\lambda^n,N})\ge x_v+Z_v(t)-\sum_{u\in N_G^-(v)}c_{uv}\bone_{\{u\in D_t^{n,N}\}},
\end{aligned}\end{equation*}
we can iteratively obtain that $D_t^{n,N}\subset\Gamma^N[\emptyset]_t$, for all $t\in[0,\infty)$ and all $N\ge1$. Therefore, $D_t^n\subset D_t$ for all $t\in[0,\infty)$, which implies $\tau_v\leq\liminf_{n\to\infty}\tau_v^n$ for all $v\in V$.

To prove the reverse inequality, it remains to show that $\tau_v\ge\limsup_{n\to\infty}\tau_v^n$ for all $v\in V$. We fix any $v_0\in V$. If $\tau_{v_0}=\infty$, then the desired inequality holds trivially. If $\tau_{v_0}<\infty$, we can perform a bottom-up induction on the backward default tree $\cT(G,v_0)$ to show that $\tau_v\ge\limsup_{n\to\infty}\tau_v^n$ for all $v\in\cT(G,v_0)$. For the base case, we assume that $v$ is a leaf node of $\cT(G,v_0)$. It necessarily holds that
\begin{equation*}\begin{aligned}
0=X_v(\tau_v)=x_v+Z_v(\tau_v)\ge x_v+Z_v(\tau_v)-\sum_{u\in N_G^-(v)}c_{uv}\lambda_{uv}(t,X_u^\lambda)=X_v^{\lambda^n},    
\end{aligned}\end{equation*}
which means $\tau_v\ge\tau_v^n$ for any $n\ge1$ and thus $\tau_v\ge\limsup_{n\to\infty}\tau_v^n$. For the induction step, let $v$ be a node such that all of its children $u$ in $\cT(G,v_0)$ satisfy $\tau_u\ge\limsup_{n\to\infty}\tau_u^n$. For any $\Delta>0$, there exists $t_\Delta\in(\tau_v,\tau_v+\Delta)$ and $\varepsilon_\Delta>0$ such that 
\begin{equation*}\begin{aligned}
x_v+Z_v(t_\Delta)-\sum_{u\in N_G^-(v)}c_{uv}\bone_{\{u\in D_{\tau_v}\}}+\varepsilon_\Delta<0.
\end{aligned}\end{equation*}
For all sufficiently large $n$, it holds that $\tau_u^n\leq\frac12(\tau_v+t_\Delta)<t_\Delta$ for any $u\in N_G^-(v)$ such that $\tau_u\leq\tau_v$, and that
\begin{equation*}\begin{aligned}
\sum_{u\in N_G^-(v)}c_{uv}\lambda_{uv}^n(t_\Delta,X_u^{\lambda^n})\ge\sum_{u\in N_G^-(v)}c_{uv}\bone_{[\tau_u^n,\infty)}(t_\Delta)-\varepsilon_\Delta    
\end{aligned}\end{equation*}
by equation \eqref{eq:lambda_conv1}. Therefore,
\begin{equation*}\begin{aligned}
X_v^{\lambda^n}(t_\Delta)&=x_v+Z_v(t_\Delta)-\sum_{u\in N_G^-(v)}c_{uv}\lambda_{uv}^n(t_\Delta,X_u^{\lambda^n})\\
&\leq x_v+Z_v(t_\Delta)-\sum_{u\in N_G^-(v)}c_{uv}\bone_{[\tau_u^n,\infty)}(t_\Delta)+\varepsilon_\Delta\\
&\leq x_v+Z_v(t_\Delta)-\sum_{u\in N_G^-(v)}c_{uv}\bone_{\{u\in D_{\tau_v}\}}+\varepsilon_\Delta<0,
\end{aligned}\end{equation*}
which implies $\tau_v^n\leq\tau_v+\Delta$. We then take $n\to\infty$ and then $\Delta\downarrow0$ to obtain that $\tau_v\ge\limsup_{n\to\infty}\tau_v$. The induction argument is then completed, and we have shown that $\tau_v=\lim_{n\to\infty}\tau_v^n$ for all $v\in V$.

To get the convergence of the solution paths $X^{\lambda_n}$, we note that
\begin{equation*}\begin{aligned}
X_v^{\lambda^n}(t)&=x_v+Z_v(t)-\sum_{u\in N_G^-(v)}c_{uv}\lambda_{uv}^n(t,X_u^{\lambda^n})\\
X_v(t)&=x_v+Z_v(t)-\sum_{u\in N_G^-(v)}c_{uv}\bone_{[\tau_u,\infty)}(t).
\end{aligned}\end{equation*}
It follows from equation \eqref{eq:lambda_conv1} and $\tau_v=\lim_{n\to\infty}\tau_v^n$ that
\begin{equation*}\begin{aligned}
&\quad d_{M_1}\Big(\sum_{u\in N_G^-(v)}c_{uv}\lambda_{uv}^n(\cdot,X_u^{\lambda^n}),\sum_{u\in N_G^-(v)}c_{uv}\bone_{[\tau_u,\infty)}(\cdot)\Big)\\ 
&\leq d_{M_1}\Big(\sum_{u\in N_G^-(v)}c_{uv}\lambda_{uv}^n(\cdot,X_u^{\lambda^n}),\sum_{u\in N_G^-(v)}c_{uv}\bone_{[\tau_u^n,\infty)}(\cdot)\Big)\\
&+d_{M_1}\Big(\sum_{u\in N_G^-(v)}c_{uv}\bone_{[\tau_u^n,\infty)}(\cdot),\sum_{u\in N_G^-(v)}c_{uv}\bone_{[\tau_u,\infty)}(\cdot)\Big)\\
&\quad\to0\quad\text{as}\quad n\to\infty.
\end{aligned}\end{equation*}
Therefore, we see that $X_v^{\lambda_n}$ converges almost surely to $X_v$ in $\cD$ as $n\to\infty$ by combining the above limit with Lemma \ref{lem:MarkedGraph_ContMap} and \cite[Lemma 12.7.3]{whitt2002stochastic}.
\end{proof}

\newpage

%%%%%%%%%%%%%%%%%%%%%%%%%%%
%\bibliographystyle{abbrvnat}
\bibliography{Main}
%%%%%%%%%%%%%%%%%%%%%%%%%%%

\bigskip\bigskip\bigskip

\end{document}